\documentclass[12pt,twoside,reqno]{amsart}
\usepackage[italian,english]{babel}
\usepackage{amssymb,latexsym}
\usepackage{amsmath,amsfonts,amsthm}
\usepackage{paralist}
\usepackage{color}

\usepackage[top=2.50cm, bottom=2.50cm, left=2.50cm, right=2.50cm]{geometry}

\numberwithin{equation}{section}
\usepackage{xspace}
\usepackage[bookmarksnumbered,colorlinks]{hyperref}
\usepackage{graphics}
\def\bb#1\eb{\textcolor{blue}
{#1}} %
\def\br#1\er{\textcolor{red}
{#1}} %
\def\bv#1\ev{\textcolor{green}
{#1}} %
\def\bc#1\ec{\textcolor{cyan}
{#1}} %

\usepackage{graphics}

\def\Xint#1{\mathchoice 
  {\XXint\displaystyle\textstyle{#1}}% 
  {\XXint\textstyle\scriptstyle{#1}}% 
  {\XXint\scriptstyle\scriptscriptstyle{#1}}% 
  {\XXint\scriptscriptstyle\scriptscriptstyle{#1}}% 
  \!\int} 
\def\XXint#1#2#3{{\setbox0=\hbox{$#1{#2#3}{\int}$} 
  \vcenter{\hbox{$#2#3$}}\kern-.5\wd0}} 
\def\-int{\Xint -}

\newcommand{\R}{\mathbb{R}}
\newcommand{\N}{\mathbb{N}}

\newcommand{\e}{\varepsilon}

\DeclareMathOperator{\dist}{dist}

\DeclareMathOperator{\supp}{supp}

\newtheorem{lem}{Lemma}
\newtheorem{thm}{Theorem}

\theoremstyle{remark}

\begin{document}

%\title{Multiple solutions for a nonlinear scalar field equation involving the Fractional Laplacian}
%\title[fractional scalar field equation]{On nonlinear scalar field equations involving the Fractional Laplacian}
\title{Mountain pass solutions for the fractional Berestycki-Lions problem}
\author{Vincenzo Ambrosio}
\address{Dipartimento di Scienze Pure e Applicate (DiSPeA),
Universit\`a degli Studi di Urbino `Carlo Bo'
Piazza della Repubblica, 13
61029 Urbino (Pesaro e Urbino, Italy)}
\email{vincenzo.ambrosio@uniurb.it}
%\address{Dipartimento di Matematica e Applicazioni ``R. Caccioppoli"\\
%         Universit\`a degli Studi di Napoli Federico II\\
 %        via Cinthia, 80126 Napoli, Italy}
%\email{vincenzo.ambrosio2@unina.it}

\begin{abstract}
We investigate the existence of least energy solutions and infinitely many solutions for the following nonlinear   
fractional equation
\begin{align*}
(-\Delta)^{s} u = g(u) \mbox{ in } \R^{N}, 
\end{align*}
%where $s\in (0,1)$, $N\geq 2$, $(-\Delta)^{s}$ is the fractional Laplacian and $g: \R \rightarrow \R$ is an odd continuous function satisfying Berestycki-Lions type assumptions. The proof is based on the symmetric mountain pass approach developed by Hirata, Ikoma and Tanaka in \cite{HIT}.
where $s\in (0,1)$, $N\geq 2$, $(-\Delta)^{s}$ is the fractional Laplacian and $g: \R \rightarrow \R$ is an odd $\mathcal{C}^{1, \alpha}$ function satisfying Berestycki-Lions type assumptions. The proof is based on the symmetric mountain pass approach developed by Hirata, Ikoma and Tanaka in \cite{HIT}.
Moreover, by combining the mountain pass approach and an approximation argument, we also prove the existence of a positive radially symmetric solution for the above problem when $g$ satisfies suitable growth conditions which make our problem fall in the so called ``zero mass" case.
\end{abstract}

\maketitle

\section{Introduction}

\noindent
In this paper we deal with the following fractional nonlinear scalar field equation with fractional diffusion
\begin{align}\label{P}
(-\Delta)^{s} u = g(u) \mbox{ in } \R^{N} 
\end{align}
with $s\in (0,1)$, $N\geq 2$, and $(-\Delta)^{s}$ is the fractional Laplacian which may be defined along a function $u$ belonging to the Schwartz space $ \mathcal{S}(\R^{N})$ of rapidly decaying functions as
$$
(-\Delta)^{s}u(x)=C_{N,s} P. V. \int_{\R^{N}} \frac{u(x)-u(y)}{|x-y|^{N+2s}} dy  \quad (x\in \R^{N}).
$$
The symbol P.V. stands for the Cauchy principal value and $C_{N,s}$ is a normalizing constant whose value can be found in  \cite{DPV}.\\
In the last decade, a great attention has been devoted to the study of elliptic equations involving fractional powers of the Laplacian. This type of operators appears in a quite natural way in many different contexts, such as, the thin obstacle problem, optimization, finance, phase transitions,  anomalous diffusion, crystal dislocation, soft thin films, flame propagation, conservation laws, ultra-relativistic limits of quantum mechanics, quasi-geostrophic flows, minimal surfaces, materials science and water waves. 
Since we cannot review the huge literature here, we refer the interested reader to \cite{DPV, MRS}, where a more extensive bibliography and an introduction to this topic are given.\\
Equation (\ref{P}) arises in the study of standing wave solutions for the fractional Schr\"odinger equation 
$$
\imath \frac{\partial \psi}{\partial t}=(-\Delta)^{s} \psi(t, x)+g(\psi)  \mbox{ in } \R^{N}\times\R
$$
when looking for standing waves, that is, solutions having the form $\psi(x, t) =e^{-\imath \omega t}u(x)$, where $\omega$ is a constant.
Such equation is fundamental in fractional quantum mechanics, and it has been introduced by Laskin \cite{Laskin1, Laskin2} as a result of expanding the Feynman path integral, from the Brownian like to the L\'evy like quantum mechanical paths. 

\noindent

Recently, the study of the fractional Schr\"odinger equation has attracted a lot of interest from many mathematicians.
Felmer et al. \cite{FQT} studied the existence, regularity and qualitative properties of positive solutions to 
$(-\Delta)^{s}u+u=f(x, u)$ when $f$ has a subcritical growth and satisfies the Ambrosetti-Rabinowitz condition; see also \cite{DPPV}. 
Coti Zelati and Nolasco \cite{CTN1, CTN2} investigated positive stationary solutions for a class of nonlinear pseudo-relativistic Schr\"odinger equations involving the operator $\sqrt{-\Delta+m^{2}}$ with $m>0$; see also \cite{A, A1, FallFelli, N.Ikoma} for related models. 
Secchi \cite{Secchi1, Secchi2} proved several existence results for some nonlinear fractional Schr\"odinger equations under the assumptions that the nonlinearity is either of perturbative type or satisfies the Ambrosetti-Rabinowitz condition. 
Molica Bisci and R\u{a}dulescu \cite{MBR} obtained the existence of multiple ground state solutions for a class of parametric fractional Schr\"odinger equations. 
%Cheng \cite{Cheng} proved the existence of bound state solutions for $(-\Delta)^{s}u+V(x)u=|u|^{p-1}u$ in which the potential $V(x)$ is unbounded and  $1<p<\frac{4s}{N}+1$.
Frank et al. \cite{FLS} established uniqueness and nondegeneracy of ground state solutions to $(-\Delta)^{s}u+u=|u|^{\alpha}u$ for all $H^{s}$-admissible powers $\alpha \in (0, \alpha^{*})$.
Pucci et al. \cite{P} dealt with the existence of multiple solutions for the nonhomogeneous fractional $p$-Laplacian equations of Schr\"odinger-Kirchhoff type,  by using the Ekeland variational principle and the Mountain Pass theorem. Figueiredo and Siciliano \cite{FS} investigated existence and multiplicity of positive solutions for a class of fractional Schr\"odinger equation by means of the Ljusternick-Schnirelmann and Morse theory.
%Torres [Torres] studied  a perturbation of (\ref{P}) when $p(x)\rightarrow +\infty$ as $|x|\rightarrow \infty$ and $f(x,u)=g(u)+h(x)$ with $f$ satisfying (AR) and $h$ is a perturbation. 
%Feng \cite{Feng} uses a concentration compactness principle in fractional Sobolev spaces to prove the existence of ground states to (\ref{P}) with $V$ asymptotically linear and $f(x,u)=|u|^{p-1}u$ with $0<p<\frac{4s}{N-2s}$.
%Liu and Gan [LG] studied a perturbed fractional Schrdinger equation with critical nonlinearity.

\smallskip
In the fundamental papers \cite{BL1, BL2}, Berestycki and Lions studied the existence and the multiplicity of nontrivial solutions to (\ref{P}) when $s=1$, that is
\begin{equation}\label{NSE}
-\Delta u = g(u) \, \mbox{ in } \, \R^{N}, 
\end{equation}
where $N \geq 3$ and $g: \R \rightarrow \R$ is an odd continuous function such that
\begin{compactenum}[({\rm BL}1)]
\item $\displaystyle{-\infty <\liminf_{t\rightarrow 0^{+}} \frac{g(t)}{t} \leq \limsup_{t\rightarrow 0^{+}} \frac{g(t)}{t}=-m<0}$;
\item $\displaystyle{-\infty <\limsup_{t\rightarrow +\infty} \frac{g(t)}{t^{\frac{N+2}{N-2}}} \leq 0}$;
\item There exists  $\xi_{0}>0$ such that  $\displaystyle{G(\xi_{0})= \int_{0}^{\xi_{0}} g(\tau) d\tau >0}$. 
\end{compactenum}
To obtain the existence of a positive solution to (\ref{NSE}), the authors  in \cite{BL1} developed a subtle Lagrange multiplier procedure which ultimately relies on the Pohozaev's identity for (\ref{NSE}). The lack of compactness due to the translational invariance of (\ref{NSE}) is regained working in the subspace of $H^{1}(\R^{N})$ of radially symmetric functions. They also proved the existence of a positive solution to \eqref{NSE} when $m=0$, the so called zero mass case. \\
In \cite{BL2} the authors showed that (\ref{NSE}) possesses infinitely many distinct bound states by applying the genus theory to the even functional $V(u)=\int_{\R^{N}} G(u) \, dx$ defined on the symmetric manifold $M=\{ u\in H^{1}_{rad}(\R^{N}): \|\nabla u\|_{L^{2}(\R^{N})}=1\}$.
%, and obtain  infinitely many  critical points of $V$.\\
Subsequently, the existence and multiplicity of solutions of closely related problems to (\ref{NSE}) have been extensively studied by many authors; see for instance \cite{ASM1, ADP, BGK, JT1, JT2}.\\ 
The question that naturally arises is whether or not the above classical existence and multiplicity results for equation (\ref{NSE}) still hold in the nonlocal framework of (\ref{P}). Chang and Wang \cite{CW} obtained the existence of a positive solution to (\ref{P}) by combining the monotonicity trick of Struwe-Jeanjean \cite{J, Struwe} and the Pohozaev identity for the fractional Laplacian.\\
The first aim of this paper is to answer the question with respect to the existence of infinitely many solutions for the equation (\ref{P}). 
%Moreover, we also consider the zero mass case, that is when $g'(0)=0$. In this way, we are able to complement and improve the result proved in \cite{CW}.
We also provide a mountain-pass characterization of least energy solutions to (\ref{P}) in the spirit of the result obtained in \cite{JT2}. In this way, we are able to complement and improve the result proved in \cite{CW}.\\
We recall that $u$ is a least energy solution to \eqref{P} if and only if
$$
I(u)=m \mbox{ and } m=\inf\left\{I(u); u\in H^{s}(\R^{N})\setminus\{0\} \mbox{ and } I'(u)=0\right\},
$$
where $I: H^{s}(\R^{N})\rightarrow \R$ is the energy functional associated to \eqref{P}, that is
$$
I(u)= \frac{1}{2} \iint_{\R^{2N}} \frac{|u(x)-u(y)|^{2}}{|x-y|^{N+2s}} \, dxdy - \int_{\R^{N}} G(u) \, dx.  
$$ 

\noindent
Now we state the main assumptions on the nonlinearity $g$.
We will assume that $g\in \mathcal{C}^{1, \alpha}(\R, \R)$, with $\alpha>\max\{0, 1-2s\}$, is odd and satisfies the following properties
%Along the paper, we will assume that $g\in \mathcal{C}^{0}(\R, \R)$ is odd and satisfies the following properties
\begin{compactenum}[({\it g}1)]
\item $\displaystyle{-\infty <\liminf_{t\rightarrow 0^{+}} \frac{g(t)}{t} \leq \limsup_{t\rightarrow 0^{+}} \frac{g(t)}{t}=-m<0}$;
\item $\displaystyle{-\infty <\limsup_{t\rightarrow +\infty} \frac{g(t)}{t^{2^{*}_{s}-1}} \leq 0}$, where $\displaystyle{2^{*}_{s}=\frac{2N}{N-2s}}$;
\item There exists  $\xi_{0}>0$ such that  $\displaystyle{G(\xi_{0})= \int_{0}^{\xi_{0}} g(\tau) d\tau >0}$. 
\end{compactenum}
We note that differently from \cite{BL1, BL2}, we don't require that $g$ is simply continuous, but a higher regularity is needed to obtain a classical solution of (\ref{P}); see \cite{CS1}.\\
Under the assumptions $(g1)$-$(g3)$, problem \eqref{P} has a variational nature, so its weak solutions can be found as critical points of the energy functional $I(u)$. 
%defined as follows
%$$
%I(u)= \frac{1}{2} \iint_{\R^{2N}} \frac{|u(x)-u(y)|^{2}}{|x-y|^{N+2s}} \, dxdy - \int_{\R^{N}} G(u) \, dx.  
%$$ 
We recall that for a weak solution of problem (\ref{P}), we mean a function $u\in H^{s}(\R^{N})$ such that 
$$
\iint_{\R^{2N}} \frac{(u(x)-u(y))}{|x-y|^{N+2s}} (\varphi(x)-\varphi(y)) \, dx dy=\int_{\R^{N}} g(u) \varphi \, dx
$$
for any $\varphi \in H^{s}(\R^{N})$.
%We recall that $u$ is a least energy solution to \eqref{P} if 
%$$
%I(u)=m \mbox{ and } m=\inf\left\{I(u); u\in H^{s}(\R^{N})\setminus\{0\} \mbox{ and } I'(u)=0\right\}.
%$$
\noindent

Our first main result is the following:
\begin{thm}\label{thmf}
Let $s\in (0,1)$ and $N\geq 2$ and $g\in \mathcal{C}^{1, \alpha}(\R, \R)$, with $\alpha>\max\{0, 1-2s\}$, be an odd function satisfying $(g1)$-$(g3)$. Then (\ref{P}) possesses a positive least energy  solution and infinitely many (possibly sign-changing) radially symmetric solutions $(u_{n})_{n \in \N}$ such that $I(u_{n})\rightarrow \infty$ as  $n \rightarrow \infty$.
%$$
%\frac{1}{2}\iint_{\R^{2N}} \frac{|u_{n}(x)-u_{n}(y)|^{2}}{|x-y|^{N+2s}} \, dx dy-\int_{\R^{N}} G(u_{n}) \, dx \rightarrow \infty \mbox{ as } n \rightarrow \infty.
%$$
Moreover, these solutions are characterized by a mountain pass and symmetric mountain-pass arguments respectively. 
%If in addition we assume that $g\in \mathcal{C}^{1, \alpha}(\R, \R)$, with $\alpha>\max\{0, 1-2s\}$, then the above solutions are $\mathcal{C}^{2, \beta}(\R^{N})$, for some $\beta\in (0,1)$.
\end{thm}
%For a weak solution of problem (\ref{P}) we mean a function $u\in H^{s}(\R^{N})$ such that 
%$$
%\iint_{\R^{2N}} \frac{(u(x)-u(y))}{|x-y|^{N+2s}} (\varphi(x)-\varphi(y)) \, dx dy=\int_{\R^{N}} g(u) \varphi \, dx
%$$
%for any $\varphi \in H^{s}(\R^{N})$.
%We recall that $u$ is a least energy solution to \eqref{P} if 
%$$
%I(u)=m \mbox{ and } m=\inf\left\{I(u); u\in H^{s}(\R^{N})\setminus\{0\} \mbox{ and } I'(u)=0\right\}.
%$$
%\noindent
It is worth noting that, a common approach to deal with fractional nonlocal problems, is to use the Caffarelli-Silvestre extension \cite{CS}, which consists to realize a given nonlocal problem through a  degenerate local problem in one more dimension by a Dirichlet to Neumann map.  
Anyway, in this work, we prefer to investigate the problem directly in $H^{s}(\R^{N})$ in order to adapt some useful methods and arguments developed in \cite{Byeon, HIT, JT2}. \\
We would like to observe that our result is in clear accordance with that for the classical local counterpart. For this reason, Theorem \ref{thmf} can be seen as the fractional version of the existence and multiplicity results given in \cite{BL1, BL2, HIT}. 

\noindent
In the second part of this paper, we use the mountain pass approach in \cite{HIT, JT2}, to prove the existence of a positive solution to (\ref{P}) in the null mass case, that is when $g$ satisfies the following properties
%By using the mountain pass approach in \cite{HIT, JT2}, we prove the existence of a positive solution to (\ref{P}) in the null mass case, that is when $g$ satisfies the following properties
\begin{compactenum}[($h_1$)]
\item $\displaystyle{\limsup_{t\rightarrow 0^{+}} \frac{g(t)}{t^{2^{*}_{s}-1}}\leq 0}$ where $2^{*}_{s}=\frac{2N}{N-2s}$;
\item $\displaystyle{\lim_{t\rightarrow +\infty} \frac{g(t)}{t^{2^{*}_{s}-1}} = 0}$;
\item There exists  $\xi_{0}>0$ such that  $\displaystyle{G(\xi_{0})= \int_{0}^{\xi_{0}} g(\tau) d\tau >0}$. 
\end{compactenum}  
We point out that the main difficulty related to the zero mass case is due to the fact that the energy of solutions of (\ref{P}) can be infinite. 
%In particular there are not so many compactness results as in the more well known positive mass case.
In the local setting, that is $s=1$, several results for zero mass problems have been established in \cite{ASM2, AFM, AP1, BGM2, BM, BL1, Struwe}.  We also recall that from a physical point of view, this type of problem is related to Yang-Mills equations, see for instance \cite{G, GNN}. 
However, differently from the classic literature, as far as we know, there is only one work \cite{A2}  concerning with zero-mass problems in non-local setting.
%Let us observe that nowadays there many papers c non-local fractional equations with superlinear nonlinearities with subcritical growth (see \cite{MRS} and references therein), but only very recently zero-mass problems have been studied also in non-local setting \cite{A2, PSY}. 
Motivated by it, here we would like to go further in this direction, by studying (\ref{P}) when $g$ is a general nonlinearity such that $g'(0)=0$.

\noindent
Our main second result can be stated as follows:
\begin{thm}\label{thmf2}
Let $s\in (0,1)$ and $N\geq 2$ and $g\in \mathcal{C}^{1, \alpha}(\R, \R)$ be an odd function satisfying $(h_1)$-$(h_3)$. Then (\ref{P}) possesses a positive radially decreasing solution.
% $u\in \mathcal{D}_{rad}^{s, 2}(\R^{N})$.
%Then (\ref{P}) possesses a nontrivial weak solution $u\in \mathcal{D}_{rad}^{s, 2}(\R^{N})$.
\end{thm}

\noindent

The proof of the above Theorem \ref{thmf2}, is obtained by combining the mountain pass approach and an approximation argument. 
%Because of the presence of nonlocal operator $(-\Delta)^{s}$, we can not use the arguments in \cite{BL1}, so  new ideas and techniques need to be introduced to overcome this difficulty.
Indeed, we show that a solution of (\ref{P}) can be approximated by a sequence of positive radially symmetric  solutions $(u_{\e})$ in $H^{s}(\R^{N})$, each one solves an approximate ``positive mass" problem $(-\Delta)^{s}u=g(u)-\e u$ in $\R^{N}$. Taking into account the mountain-pass characterization of least energy solution of (\ref{P}) given in Theorem \ref{thmf}, we are capable to obtain lower and upper bounds for the mountain pass critical levels $b^{\e}_{mp}$, which can be estimated independently from $\e$, when $\e$ is sufficiently small. This allows us to pass to the limit in $(-\Delta)^{s}u_{\e}=g(u_{\e})-\e u_{\e}$ in $\R^{N}$ as $\e \rightarrow 0$, and to find a nontrivial solution $u$ to (\ref{P}). It is worth to point out that our proof is different from the ones given in \cite{BL1}, and that it works even when $s=1$.
\smallskip

\noindent
To our knowledge all results presented here are new.
The paper is organized as follows. Section $2$ contains some preliminary results about fractional Sobolev spaces and functional setting. In Section $3$ we prove the existence of infinitely many solutions to (\ref{P}) by using an auxiliary functional $\tilde{I}(\theta, u)$ on the augmented space $\R \times H^{s}_{rad}(\R^{N})$. In Section $4$ we obtain the existence of a least energy solution to (\ref{P}) characterized by a mountain-pass argument, and we investigate regularity and symmetry of this solution. Finally, by using an approximation argument, we get a positive radially symmetric positive solution to (\ref{P}) in the null mass case.

\section{Preliminaries and functional setting}

\noindent
In this section we collect some preliminary results which will be useful along the paper. 

\subsection{Fractional Sobolev spaces}
We recall some useful facts of the fractional Sobolev spaces; for more details we refer to \cite{DPV}.\\
For any $s\in (0,1)$ we denote by $\mathcal{D}^{s,2}(\R^{N})$ the completion of $\mathcal{C}^{\infty}_{0}(\R^{N})$ with respect to the so-called Gagliardo seminorm
$$
[u]^{2}_{H^{s}(\R^{N})} =\iint_{\R^{2N}} \frac{|u(x)-u(y)|^{2}}{|x-y|^{N+2s}} \, dx \, dy.
$$
That is 
$$
\mathcal{D}^{s,2}(\R^{N})=\{u\in L^{2^{*}_{s}}(\R^{N}): [u]_{H^{s}(\R^{N})}<\infty\},
$$
where 
$$
2^{*}_{s}=\frac{2N}{N-2s}
$$
is the critical Sobolev exponent.
We also use the notation
$$
\langle u, v \rangle_{\mathcal{D}^{s, 2}(\R^{N})}=\iint_{\R^{2N}} \frac{(u(x)-u(y))}{|x-y|^{N+2s}}(v(x)-v(y)) \, dx \, dy
$$
for all $u, v\in \mathcal{D}^{s, 2}(\R^{N})$.
Now, we define the standard fractional Sobolev space
$$
H^{s}(\R^{N})= \left\{u\in L^{2}(\R^{N}) : \frac{|u(x)-u(y)|}{|x-y|^{\frac{N+2s}{2}}} \in L^{2}(\R^{2N}) \right \}
$$
endowed with the natural norm 
$$
\|u\|_{H^{s}(\R^{N})} = \sqrt{[u]_{H^{s}(\R^{N})}^{2} + \|u\|_{L^{2}(\R^{N})}^{2}}
$$
%where the term
%$$
%[u]^{2}_{H^{s}(\R^{N})} =\iint_{\R^{2N}} \frac{|u(x)-u(y)|^{2}}{|x-y|^{N+2s}} \, dx \, dy
%$$
%is the so-called Gagliardo seminorm of $u$. \\
Let us denote by 
$$
\langle u, v \rangle_{H^{s}(\R^{N})} =\langle u, v \rangle_{\mathcal{D}^{s, 2}(\R^{N})}+\int_{\R^{N}} u v \, dx
$$
for any $u, v \in H^{s}(\R^{N})$.
%In order to we will write $[\cdot]$ for $[u]^{2}_{H^{s}(\R^{N})}$ and $|u|_{q}$ to denote $||u||_{L^{q}(\R^{N})}$
For the reader's convenience, we review the main embedding result for this class of fractional Sobolev spaces. 
\begin{thm}\label{Sembedding}
Let $s\in (0,1)$ and $N>2s$. Then $H^{s}(\R^{N})$ is continuously embedded in $L^{q}(\R^{N})$ for any $q\in [2, 2^{*}_{s}]$ and compactly in $L^{q}_{loc}(\R^{N})$ for any $q\in [2, 2^{*}_{s})$. 
\end{thm}

\noindent
Now we introduce 
$$
H^{s}_{rad}(\R^{N})=\{u\in H^{s}(\R^{N}): u(x)=u(|x|)\}
$$
the space of radial functions in $H^{s}(\R^{N})$.
We recall the following compactness result due to Lions \cite{Lions}:
\begin{thm}\cite{Lions}\label{Lions}
Let $s\in (0,1)$ and $N\geq 2$. Then $H^{s}_{rad}(\R^{N})$ is compactly embedded in $L^{q}(\R^{N})$ for any $q\in (2, 2^{*}_{s})$.
\end{thm}

Finally, we recall the following useful lemmas:

\begin{lem}\cite{BL1}\label{radlem}
Let $u\in L^{t}(\R^{N})$, $1\leq t<\infty$ be a nonnegative radially decreasing function (that is $0\leq u(x)\leq u(y)$ if $|x|\geq |y|$). Then 
\begin{equation}\label{CO}
|u(x)|\leq \left(\frac{N}{\omega_{N-1}}\right)^{\frac{1}{t}} |x|^{-\frac{N}{t}} \|u\|_{L^{t}(\R^{N})} \mbox{ for any } x\in \R^{N}\setminus \{0\},
\end{equation}
where $\omega_{N-1}$ is the Lebesgue measure of the unit sphere in $\R^{N}$.
\end{lem}

\begin{lem}\cite{ADP, BL1}\label{strauss}
Let $P$ and $Q:\R\rightarrow\R$ be a continuous functions satisfying
\begin{equation*}
\lim_{t\rightarrow +\infty} \frac{P(t)}{Q(t)}=0.
\end{equation*}
Let $(v_{k})$, $v$ and $w$ be measurable functions from $\R^{N}$ to $\R$, with $w$ bounded, such that 
\begin{align*}
&\sup_{k} \int_{\R^{N}} |Q(v_{k}(x)) w| \, dx <+\infty, \\
&P(v_{k}(x))\rightarrow v(x) \mbox{ a.e. in } \R^{N}. 
\end{align*}
Then $\|(P(v_{k})-v)w\|_{L^{1}(\mathcal{B})}\rightarrow 0$, for any bounded Borel set $\mathcal{B}$. \\
Moreover, if we have also
\begin{equation*}
\lim_{t\rightarrow 0} \frac{P(t)}{Q(t)}=0,
\end{equation*}
and
\begin{equation*}
\lim_{|x| \rightarrow \infty} \sup_{k\in \N} |v_{k}(x)|=0,
\end{equation*}
then $\|(P(v_{k})-v)w\|_{L^{1}(\R^{N})}\rightarrow 0$.
\end{lem}

\begin{lem}\cite{CW}\label{Strauss}
Let  $(X, \|\cdot\|)$ be a Banach space such that $X$ is embedded continuously and compactly into $L^{q}(\R^{N})$ for $q\in [q_{1}, q_{2}]$ and $q\in (q_{1}, q_{2})$ respectively, where $q_{1}, q_{2}\in (0, \infty)$.
Assume that $(u_{k})\subset X$, $u: \R^{N} \rightarrow \R$ is a measurable function and $P\in C(\R, \R)$ is such that
\begin{compactenum}[$(i)$]
\item $\displaystyle{\lim_{|t|\rightarrow 0} \frac{P(t)}{|t|^{q_{1}}}=0}$, \\
\item $\displaystyle{\lim_{|t|\rightarrow \infty} \frac{P(t)}{|t|^{q_{2}}}=0}$,\\
\item $\displaystyle{\sup_{k\in \N} \|u_{k}\|<\infty}$,\\
\item $\displaystyle{\lim_{k \rightarrow \infty} P(u_{k}(x))=u(x)} \mbox{ for a.e. } x\in \R^{N}$.
\end{compactenum}
Then, up to a subsequence, we have
$$
\lim_{k\rightarrow \infty} \|P(u_{k})-u\|_{L^{1}(\R^{N})}=0.
$$
\end{lem}

\section{Infinitely many solutions to (\ref{P})}
\subsection{The modification of $g$ and introduction of a penalty function}
In order to give the proof of Theorem \ref{thmf}, we redefine the nonlinearity $g$ as follows: 
\begin{compactenum}[$(i)$]
\item If $g(t)> 0$ for all $t\geq \xi_{0}$, we simply extend $g$ to the negative axis: 
\begin{equation*}
\tilde{g}(t)= 
\left\{
\begin{array}{ll}
g(t) &\mbox{ for } t\geq 0 \\
-g(-t)    & \mbox{ for } t<0;
\end{array}
\right.
\end{equation*} 
\item If there exists $t_{0}> \xi_{0}$ such that $g(t_{0})=0$, we put 
\begin{equation*}
\tilde{g}(t)= 
\left\{
\begin{array}{ll}
g(t) &\mbox{ for } t\in [0, t_{0}] \\
0     & \mbox{ for } t> t_{0} \\
-\tilde{g}(-t)    & \mbox{ for } t<0.
\end{array}
\right.
\end{equation*} 
\end{compactenum}
Then $\tilde{g}$ satisfies $(g1)$, $(g3)$ and 
\begin{equation}\tag{$g2'$}
\lim_{t\rightarrow \infty} \frac{\tilde{g}(t)}{|t|^{\frac{N+2s}{N-2s}}} =0.
\end{equation}
%Moreover, if $u$ is a solution to $(-\Delta)^{s}u=\tilde{g}(u)$ in $\R^{N}$, then $u$ satisfies $-t_{0}\leq u\leq t_{0}$ in $\R^{N}$ (for instance, in the case $(ii)$ above, to see that $u\leq t_{0}$, it is enough to multiply by $(u-t_{0})^{+}$ both members of $(-\Delta)^{s}u=\tilde{g}(u)$ so that we find $[(u-t_{0})^{+}]^{2}_{H^{s}(\R^{N})}\leq 0$), that is $u$  is also a solution to \eqref{P}.
Moreover, by the weak maximum principle \cite{CS1}, any solution to $(-\Delta)^{s}u=\tilde{g}(u)$ in $\R^{N}$ is also a solution to (\ref{P}). Indeed, in the case $(ii)$ above, any solution to $(-\Delta)^{s}u=\tilde{g}(u)$ in $\R^{N}$ satisfies $-t_{0}\leq u(x)\leq t_{0}$ for any $x\in\R^{N}$. To prove this, we use $(u-t_{0})^{+}\in H^{s}(\R^{N})$ as test function in the weak formulation of $(-\Delta)^{s}u=\tilde{g}(u)$ in $\R^{N}$, and we get
$$
\iint_{\R^{2N}} \frac{(u(x)-u(y))}{|x-y|^{N+2s}} ((u(x)-t_{0})^{+}-(u(y)-t_{0})^{+})\, dx dy=\int_{\R^{N}} \tilde{g}(u) (u(x)-t_{0})^{+} \,dx.
$$
From the definition of $\tilde{g}$, it follows that $\int_{\R^{N}} \tilde{g}(u) (u(x)-t_{0})^{+} \,dx\leq 0$.
Then, recalling that $(x-y)(x^{+}-y^{+})\geq |x^{+}-y^{+}|^{2}$ for all $x, y\in \R$,
we can see that 
$$
\iint_{\R^{2N}} \frac{|(u(x)-t_{0})^{+}-(u(y)-t_{0})^{+}|^{2}}{|x-y|^{N+2s}} \, dx dy\leq 0,
$$
which implies that $[(u-t_{0})^{+}]^{2}_{H^{s}(\R^{N})}\leq 0$. Since $u\in H^{s}(\R^{N})$, we deduce that $u\leq t_{0}$ in $\R^{N}$. The other inequality is obtained in similar way by using $(u+t_{0})^{-}$ as test function.
Here we used the notations $x^{+}=\max\{x, 0\}$ and $x^{-}=\min\{x, 0\}$.\\
Therefore, from now on, we will tacitly write $g$ instead of $\tilde{g}$, and we will assume that $g$ satisfies $(g1)$, $(g2')$ and $(g3)$.
%and 
%\begin{equation}
%\lim_{t\rightarrow \infty} \frac{g(t)}{|t|^{\frac{N+2s}{N-2s}}} =0. 
%\end{equation}

\noindent
Now we introduce a penalty function to construct an auxiliary function.
For $t\geq 0$ we define
$$
f(t) =\max\Bigl\{0, \frac{1}{2}mt+g(t) \Bigr\}
$$
and 
$$
h(t) =t^{p} \sup_{0 <\tau \leq t} \frac{f(\tau)}{\tau^{p}}
$$
where $p$ is a positive number such that $1<p<\frac{N+2s}{N-2s}$.\\
Note that $f$ and $h$ are well defined in view of  $(g1)$. We extend $h$ as an odd function on $\R$ and we set 
$$
H(t) =\int_{0}^{t} h(\tau) d\tau.
$$
Next, we prove the following result whose proof follows the lines in \cite{HIT}. For reader's convenience we give the proof here. 
\begin{lem}\label{lemh}
The above function $h$ satisfies the following properties:
\begin{compactenum}[(h1)]
\item $h\in \mathcal{C}(\R, \R)$, $h(t)\geq 0$ and $h(-t)=-h(t)$ for all $t\geq 0$. 
\item There exists $\beta>0$ such that $h=0=H$ on $[-\beta, \beta]$.
\item For all $t\in \R$
$$
\frac{1}{2}mt^{2}+g(t)t\leq h(t)t\, \mbox{ and } \,\frac{1}{4}mt^{2}+G(t)\leq H(t).
$$
\item $ \displaystyle{\lim_{|t| \rightarrow \infty} \frac{h(t)}{|t|^{\frac{N+2s}{N-2s}}}=0}$. 
\item $h$ satisfies the following Ambrosetti-Rabinowitz condition:
$$
0\leq (p+1)H(t)\leq h(t)t \,\mbox{ for all } t\in \R.
$$
\item If $(u_{k})$ is a bounded sequence in $H^{s}_{rad}(\R^{N})$ then
$$
\lim_{k\rightarrow \infty} \int_{\R^{N}} h(u_{k}) u_{k}\, dx=\int_{\R^{N}} h(u) u \,dx.
$$
\end{compactenum}
\end{lem}
\begin{proof}
Let us observe that $(h1)$, $(h2)$ and $(h3)$ follow by $(g1)$ and the definition of $f$, $h$ and $H$.\\
$(h4)$ Firstly, we remark that
\begin{align*}
\frac{h(t)}{t^{2^{*}_{s}-1}}&=t^{-(2^{*}_{s}-1-p)} \sup_{0<\tau \leq t} \frac{f(\tau)}{\tau^{p}}\\
&=\sup_{0<\tau \leq t} \frac{f(\tau)}{\tau^{2^{*}_{s}-1}} \frac{\tau^{2^{*}_{s}-1-p}}{t^{2^{*}_{s}-1-p}}
\end{align*}
Since $f$ satisfies $(g2')$, for any $\varepsilon>0$ there exists $t_{\varepsilon}>0$ such that
$$
\Bigl|\frac{f(\tau)}{\tau^{2^{*}_{s}-1}}\Bigr|< \varepsilon \mbox{ for all } \tau \geq t_{\varepsilon}.
$$
Set $\displaystyle{C_{\varepsilon}=\sup_{0<\tau\leq t_{\varepsilon}} \Bigl|\frac{f(\tau)}{\tau^{2^{*}_{s}-1}}\Bigr|}$.
Then we have
\begin{align*}
\frac{h(t)}{t^{2^{*}_{s}-1}} &\leq \max \left\{\sup_{0<\tau \leq \tau_{\varepsilon}} \Bigl |\frac{f(\tau)}{\tau^{2^{*}_{s}-1}} \Bigr| \frac{\tau_{\varepsilon}^{2^{*}_{s}-1-p}}{t^{2^{*}_{s}-1-p}},  \sup_{\tau_{\varepsilon}\leq \tau\leq t_{\varepsilon}} \Bigl|\frac{f(\tau)}{\tau^{2^{*}_{s}-1}}\Bigr| \right\} \\
&\leq \max \left\{ \frac{C_{\varepsilon}\tau_{\varepsilon}^{2^{*}_{s}-1-p}}{t^{2^{*}_{s}-1-p}} , \varepsilon \right\} 
\end{align*}
so we deduce that
$$
0\leq \limsup_{t\rightarrow +\infty} \frac{h(t)}{t^{2^{*}_{s}-1}}\leq \varepsilon.
$$
$(h5)$ By the definition of $h$ and $H$ we have
\begin{align*}
(p+1)H(t)-h(t)t&=\int_{0}^{t} [(p+1)h(\tau)-h(t)] d\tau \\
&=\int_{0}^{t} \Bigl[(p+1)\tau^{p} \sup_{0< \xi \leq \tau} \frac{f(\xi)}{\xi^{p}}-t^{p} \sup_{0< \xi \leq t} \frac{f(\xi)}{\xi^{p}}\Bigr] d\tau \\
&\leq \sup_{0< \xi \leq t} \frac{f(\xi)}{\xi^{p}} \int_{0}^{t} (p+1)\tau^{p}-t^{p}d\tau=0.
\end{align*}
Now we prove $(h6)$.
Let $(u_{k})$ be a bounded sequence in $H^{s}_{rad}(\R^{N})$.
By using $(h2)$ and $(h4)$ we know that
$$
\lim_{|t|\rightarrow 0} \frac{h(t)t}{|t|^{2}}=0
$$
and
$$
\lim_{|t|\rightarrow \infty} \frac{h(t)t}{|t|^{2^{*}_{s}}}=0.
$$
Then we can apply Lemma \ref{Strauss} with $P(t)=h(t)t$, $q_{1}=2$ and $q_{2}=2^{*}_{s}$, to deduce that
$$
\lim_{k \rightarrow \infty} \int_{\R^{N}} h(u_{k}) u_{k} \, dx=\int_{\R^{N}} h(u) u \,dx.
$$

\end{proof}

%\begin{remark}
%The property $(h6)$ will be fundamental to obtain the compactness of bounded Palais-Smale sequences in Theorem \ref{psc} and Theorem \ref{thmcomp} below.
%\end{remark}

\subsection{Comparison between functionals}

Let us consider the following norm on $H^{s}(\R^{N})$ 
$$
\|u\|^{2} =[u]_{H^{s}(\R^{N})}^{2}+\frac{m}{2}\|u\|_{L^{2}(\R^{N})}^{2}
$$
which is equivalent to the standard norm $\|\cdot \|_{H^{s}(\R^{N})}$ defined in Section $2$.\\
Let us introduce the following functionals $I: H_{rad}^{s}(\R^{N})\rightarrow \R$ and  $J: H_{rad}^{s}(\R^{N})\rightarrow \R$, by setting
$$
I(u)=\frac{1}{2} \|u\|^{2}-\int_{\R^{N}} \frac{m}{4}u^{2}+G(u) dx 
$$
and 
$$
J(u)= \frac{1}{2} \|u\|^{2}-2\int_{\R^{N}} H(u) dx.
$$
By the growth assumptions on $g$, it is easy to check that $I$ and $J$ are well defined and that they are $\mathcal{C}^{1}(H^{s}_{rad}(\R^{N}))$-functionals.
Clearly, critical points of $I$ and $J$ are weak solutions to (\ref{P}) and $\displaystyle{(-\Delta)^{s}u+\frac{m}{2}u=2 h(u)}$ in $\R^{N}$, respectively.

\noindent
In what follows, we show that $I$ and $J$ have a symmetric mountain pass geometry:
\begin{lem}\label{lem4.1}
\noindent
\begin{compactenum}[$(i)$]
\item $I(u)\geq J(u)$ for all $u\in H^{s}_{rad}(\R^{N})$; 
\item There are $\delta>0$ and $\rho>0$ such that
\begin{align*}
&I(u), J(u)\geq \delta>0 \mbox{ for } \|u\|=\rho \\
&I(u), J(u)\geq 0 \mbox{ for } \|u\|\leq \rho;
\end{align*}
\item For any $n\in \N$ there exists an odd continuous map $\gamma_{n}: \mathbb{S}^{n-1}\rightarrow H^{s}_{rad}(\R^{N})$ such that
$$
I(\gamma_{n}(\sigma)), J(\gamma_{n}(\sigma))<0 \mbox{ for all } \sigma \in \mathbb{S}^{n-1}.
$$
\end{compactenum}
\end{lem}
\begin{proof}
$(i)$ follows by $(h3)$ of Lemma \ref{lemh}. \\
$(ii)$ By Lemma \ref{lemh} there exists $C>0$ such that
\begin{equation*}
H(t)\leq C|t|^{2^{*}_{s}} \mbox{ for all } t\in \R.
\end{equation*}
Then, by using Sobolev's embedding, we can see that
\begin{equation*}
J(u)\geq \frac{1}{2}\|u\|^{2}-C\|u\|_{L^{2^{*}_{s}}(\R^{N})}^{2^{*}_{s}}\geq \|u\|^{2}\Bigl[\frac{1}{2}-C'_{*}\|u\|^{2^{*}_{s}-2}\Bigr].
\end{equation*}
Since $2^{*}_{s}>2$, we can find $\delta, \rho>0$ such that $J(u)\geq \delta$ for  $\|u\|=\rho$, and 
$J(u)\geq 0$ if  $\|u\|\leq \rho$. 
From $(i)$, we deduce that $(ii)$ holds.\\
$(iii)$ 
For $n\in \N$ we consider the polyhedron in $\R^{n}$ defined by
$$
\mathbb{S}^{n-1}=\Bigl\{\sigma=(\sigma_{1}, \dots, \sigma_{n})\in \R^{n}: \sum_{j=1}^{n}|\sigma_{j}|=1\Bigr \}.
$$
By using Theorem $10$ in \cite{BL2}, for any $n\in \N$, there exists an odd continuous map $\pi_{n}: \mathbb{S}^{n-1}\rightarrow H^{1}(\R^{N})$ such that:
\begin{compactenum}[$(i)$]
\item $\pi_{n}(\sigma)$ is radial for all $\sigma\in \mathbb{S}^{n-1}$;
\item $0\notin \pi_{n}(\mathbb{S}^{n-1})$;
\item $\displaystyle{\int_{\R^{N}} G(\pi_{n}(\sigma)) dx \geq 1 \mbox{ for any } \sigma\in \mathbb{S}^{n-1}}$.
\end{compactenum}

\noindent
Since $H^{1}(\R^{N})\subset H^{s}(\R^{N})$ and $\pi_{n}(\mathbb{S}^{n-1})$ is compact, there exists $M>0$ such that 
$$
\|\pi_{n}(\sigma)\|_{H^{s}(\R^{N})}\leq M \mbox{ for any } \sigma\in \mathbb{S}^{n-1}.
$$

\noindent
Now, let us define $\psi^{t}_{n}(\sigma)(x)=\pi_{n}(\sigma)(\frac{x}{t})$ with $t\geq 1$.\\
Then we have
\begin{align*}
I(\psi^{t}_{n}(\sigma))&=t^{N-2s}\frac{1}{2} [\pi_{n}(\sigma)]^{2}_{H^{s}(\R^{N})}-t^{N}\int_{\R^{N}} G(\pi_{n}(\sigma)) dx \\
&\leq t^{N-2s} \Bigl[\frac{M}{2}-t^{2s}\Bigr] \rightarrow -\infty \mbox{ as } t\rightarrow +\infty.
\end{align*}
Therefore we can chose $\overline{t}$ such that $I(\psi^{\overline{t}}_{n}(\sigma))<0$ for all $\sigma \in \mathbb{S}^{n-1}$, 
and by setting $\gamma_{n}(\sigma)(x):=\psi^{\overline{t}}_{n}(\sigma)(x)$, we can infer that $\gamma_{n}$ satisfies the required properties for $I$. Taking into account $(i)$, we can conclude the proof.
\end{proof}

\noindent
Differently from $I(u)$, the comparison functional $J(u)$ satisfies the following compactness property:
\begin{thm}\label{psc}
The functional $J$ satisfies the Palais-Smale condition.
\end{thm}
\begin{proof}
Let $c\in \R$ and let $(u_{k})\subset H^{s}_{rad}(\R^{N})$ such that
\begin{align}\label{psh}
J(u_{k})\rightarrow c \mbox{ and } J'(u_{k}) \rightarrow 0.
\end{align}
By using $(h5)$ of Lemma \ref{lemh}, we can see
\begin{align*}
J(u_{k})-\frac{J'(u_{k})u_{k}}{p+1}&=\Bigl(\frac{1}{2}-\frac{1}{p+1}\Bigr)\|u_{k}\|^{2}-2\int_{\R^{N}} \left[H(u_{k})-\frac{1}{p+1} h(u_{k})u_{k} \right] \, dx\\ 
&\geq \Bigl(\frac{1}{2}-\frac{1}{p+1}\Bigr)\|u_{k}\|^{2}
\end{align*}
so we can deduce that $(u_{k})$ is bounded in $H^{s}_{rad}(\R^{N})$.

\noindent
Then, by using Lemma \ref{Lions}, we may assume, up to a subsequence, that
\begin{align}\label{convergencesforu}
& u_{k} \rightharpoonup u \mbox{ in } H^{s}_{rad}(\R^{N}), \nonumber \\
& u_{k} \rightarrow u \mbox{ in } L^{q}(\R^{N}) \quad \forall q\in (2, 2^{*}_{s}),\\
& u_{k} \rightarrow u  \mbox{ a.e. } \R^{N}.\nonumber
\end{align}
%By using the fact that $J'(u_{k})\rightarrow 0$ 
Now, by using $(h1)$ and $(h4)$ in Lemma \ref{lemh} and (\ref{convergencesforu}), we can apply the first part of Lemma \ref{strauss}
%a variant of Strauss' compactness Lemma (see Proposition $2.5$ in \cite{ADP}) 
with $P(t)=h(t)$ and $Q(t)=|t|^{2^{*}_{s}-1}$, to infer that for any $\varphi \in \mathcal{C}^{\infty}_{0}(\R^{N})$
%$$
%\lim_{k \rightarrow \infty} \int_{\R^{N}} g_{i}(u_{k}(x)) \varphi(x) dx=\int_{\R^{N}} g_{i}(u(x)) \varphi(x) dx.
%$$
%and
\begin{align}\label{vincenzo}
\lim_{k \rightarrow \infty} \int_{\R^{N}} h(u_{k}(x)) \varphi(x) dx=\int_{\R^{N}} h(u(x)) \varphi(x) dx.
\end{align}
Putting together (\ref{psh}), (\ref{convergencesforu}) and (\ref{vincenzo}) we obtain
\begin{align*}
J'(u)\varphi=J'(u_{k})\varphi-\left[\langle u_{k}-u, \varphi \rangle_{H^{s}(\R^{N})}-2\int_{\R^{N}} (h(u_{k})-h(u)) \varphi \,dx\right]\rightarrow 0
\end{align*}
for any $\varphi \in \mathcal{C}^{\infty}_{0}(\R^{N})$.
Since $\mathcal{C}^{\infty}_{0}(\R^{N})$ is dense in $H^{s}_{rad}(\R^{N})$, it follows that
$$
J'(u) \varphi=0 \mbox{ for all } \varphi \in H^{s}_{rad}(\R^{N}).
$$
Moreover,  $J'(u)u=0$, which implies that $\|u\|^{2}=2\int_{\R^{N}} h(u) u \,dx$.\\
By using $(h6)$ of Lemma \ref{lemh}, we know that 
\begin{equation}\label{29}
\lim_{k \rightarrow \infty} \int_{\R^{N}} h(u_{k}) u_{k} \,dx=\int_{\R^{N}} h(u) u\, dx.
\end{equation}
Taking into account the boundedness of $(u_{k})$ and (\ref{psh}), we get $J'(u_{k})u_{k}\rightarrow 0$, which together with (\ref{29}) implies that $	\|u_{k}\|\rightarrow \|u\|$ as $k \rightarrow \infty$. 
Therefore, we can conclude that $u_{k} \rightarrow u$ in $H^{s}_{rad}(\R^{N})$ as $k \rightarrow \infty$.

\end{proof}

\noindent
Now we define minimax values of $I$ and $J$ by using maps $(\gamma_{n})$ in Lemma \ref{lem4.1}.
For any $n\in \N$, we define $b_{n}$ and $c_{n}$ as follows:
$$
b_{n}=\inf_{\gamma\in \Gamma_{n}} \max_{\sigma \in D_{n}} I(\gamma(\sigma)),
$$
$$
c_{n}=\inf_{\gamma\in \Gamma_{n}} \max_{\sigma \in D_{n}} J(\gamma(\sigma))
$$
where $D_{n}=\{\sigma \in \R^{n}: |\sigma|\leq 1\}$ and
$$
\Gamma_{n}=\{\gamma\in C(D_{n}, H^{s}_{rad}(\R^{N})): \gamma \mbox{ is odd and } \gamma=\gamma_{n} \mbox{ on } \mathbb{S}^{n-1}\}.
$$
The values $b_{n}$ and $c_{n}$ satisfy the following properties. 
\begin{lem}\label{lem4.4}
\noindent
\begin{compactenum}[(i)]
\item $\Gamma_{n}\neq \emptyset$ for any $n\in \N$;
\item $0<\delta\leq c_{n}\leq b_{n}$ for any $n\in \N$,
where $\delta$ appears in Lemma \ref{lem4.1}.
\end{compactenum}
\end{lem}
\begin{proof}
$(i)$ Let us define
\begin{equation*}
\tilde{\gamma}_{n}(\sigma):=
\left\{
\begin{array}{ll}
|\sigma| \gamma_{n}(\frac{\sigma}{|\sigma|}) &\mbox{ if } \sigma \in D_{n}\setminus \{0\}\\
0 &\mbox{ for } \sigma=0.
\end{array}
\right.
\end{equation*}
Then, it is clear that $\tilde{\gamma}_{n}\in \Gamma_{n}$.\\
$(ii)$ Since $I(u)\geq J(u)$ in view  of Lemma \ref{lem4.1}, it holds $c_{n}\leq b_{n}$.
The property $\delta \leq c_{n}$ follows from the fact that
$$
\{u\in H^{s}_{rad}(\R^{N}): \|u\|=\rho \} \cap \gamma(D_{n})\neq \emptyset \mbox{ for all } \gamma \in \Gamma_{n}.
$$
\end{proof}

\begin{lem}\label{lem3.2}
\noindent
\begin{compactenum}[(i)]
\item The value $c_{n}$ is a critical value of $J$.
\item $c_{n} \rightarrow \infty$ as $n \rightarrow \infty$. 
\end{compactenum}
\end{lem}
\begin{proof}
$(i)$ follows by Lemma \ref{psc}.\\
$(ii)$ Set
$$
\Sigma_{n}\!\!=\!\!\left \{h\in C(\overline{D_{m}\setminus Y}): h\in \Gamma_{m}, m\geq n, Y\in \mathcal{E}_{m} \mbox{ and } genus(Y)\leq m-n \right\}
$$
where $\mathcal{E}_{m}$ is the family of closed sets $A\subset \R^{m}\setminus \{0\}$ such that $-A=A$ and $genus(A)$ is the Krasnoselski genus of $A$.
Now we define another sequence of minimax values by setting
$$
d_{n}=\inf_{A\in \Sigma_{n}} \max_{u\in A} J(u).
$$
Then $d_{n}\leq d_{n+1}$ for all $n\in \N$, and $d_{n}\leq c_{n}$ for all $n\in \N$. 
Since $J$ satisfies the Palais-Smale condition, we can proceed as in the proof of Proposition $9.33$ in \cite{Rab} to infer that $d_{n}\rightarrow +\infty$ as $n \rightarrow \infty$.

\end{proof}

\noindent
Let us observe that $(ii)$ of Lemma \ref{lem4.4} and $(ii)$ of Lemma \ref{lem3.2} yield
\begin{equation}\label{bn}
b_{n}>0 \mbox{ for all } n\in \N \mbox{ and } \lim_{n \rightarrow \infty} b_{n}=\infty.
\end{equation}
In the next section we will prove that $b_{n}$ are critical values of $I(u)$.

\subsection{An auxiliary functional on augmented space}

Let us introduce the following functional 
$$
\tilde{I}(\theta, u)=\frac{1}{2}e^{(N-2s)\theta} [u]_{H^{s}(\R^{N})}^{2}-e^{N \theta}\int_{\R^{N}} G(u) dx
$$
for $(\theta, u)\in \R\times H^{s}_{rad}(\R^{N})$.\\
We endow $\R\times H^{s}_{rad}(\R^{N})$ with the norm
$$
\|(\theta, u)\|_{\R\times H^{s}(\R^{N})} =\sqrt{|\theta|^{2}+\|u\|^{2}}.
$$
Let us point out that $\tilde{I}\in \mathcal{C}^{1}(\R\times H^{s}_{rad}(\R^{N}), \R)$, $\tilde{I}(0,u)=I(u)$ and being 
\begin{align*}
&\iint_{\R^{2N}} \frac{|u(e^{-\theta}x)-u(e^{-\theta}y)|^{2}}{|x-y|^{N+2s}} dx dy=e^{(N-2s) \theta}[u]_{H^{s}(\R^{N})}^{2}\\
&\int_{\R^{N}} G(u(e^{-\theta}x)) dx=e^{N\theta}\int_{\R^{N}} G(u(x)) dx
\end{align*}
we have 
\begin{align}\label{4.2}
\tilde{I}(\theta, u(x))=I(u(e^{-\theta}x)) \mbox{ for all } \theta \in \R, u\in H^{s}_{rad}(\R^{N}).
\end{align}
We also define minimax values $\tilde{b}_{n}$ for $\tilde{I}(\theta, u)$ by
$$
\tilde{b}_{n}=\inf_{\tilde{\gamma}\in \tilde{\Gamma}_{n}} \max_{\sigma \in D_{n}} \tilde{I}(\tilde{\gamma}(\sigma)),
$$
where 
\begin{align*}
\tilde{\Gamma}_{n}=\{\tilde{\gamma}\in C(D_{n}, \R \times H^{s}_{rad}(\R^{N})): \tilde{\gamma}(\sigma)=(\theta(\sigma), \eta(\sigma)) \mbox{ satisfies } \\
(\theta(-\sigma), \eta(-\sigma))= (\theta(\sigma), -\eta(\sigma)) \mbox{ for all } \sigma \in D_{n} \\
(\theta(\sigma), \eta(\sigma))= (0, \gamma_{n}(\sigma)) \mbox{ for all }  \sigma \in \mathbb{S}^{n-1}\}.
\end{align*}
Through the modified functional $\tilde{I}(\theta, u)$ we aim to show that $b_{n}$ are critical values for $I(u)$.

\noindent
We begin proving that
\begin{lem}\label{lembn}
$\tilde{b}_{n}=b_{n}$ for all $n\in \N$.
\end{lem}
\begin{proof}
Let us observe that $(0, \gamma(\sigma))\in \tilde{\Gamma}_{n}$ for any $\gamma\in \Gamma_{n}$, so we can see that $\Gamma_{n} \subset \tilde{\Gamma}_{n}$. 
Since $\tilde{I}(0, u)=I(u)$, by the definitions of $b_{n}$ and $\tilde{b}_{n}$, it follows that $\tilde{b}_{n}\leq b_{n}$ for all $n\in \N$.\\
Now, we take $\tilde{\gamma}(\sigma)=(\theta(\sigma), \gamma(\sigma))\in \tilde{\Gamma}_{n}$ and we put $\gamma(\sigma)=\eta(\sigma)(e^{-\theta(\sigma)}x)$. Then it is easy to see that $\gamma\in \Gamma_{n}$, and  by using (\ref{4.2}), $I(\gamma(\sigma))=\tilde{I}(\tilde{\gamma}(\sigma))$ for all $\sigma \in D_{n}$. As a consequence we have $\tilde{b}_{n}\geq b_{n}$ for all $n\in \N$.

\end{proof}

\noindent
At this point, we show that the functional $\tilde{I}(\theta, u)$ admits a Palais-Smale sequence in $\R\times H^{s}_{rad}(\R^{N})$
with a property related to the fractional Pohozaev identity \cite{CW}:
\begin{equation}\label{pepsiboom}
\frac{N-2s}{2}\int_{\R^{N}} |(-\Delta)^{\frac{s}{2}} u|^{2} \, dx=N \int_{\R^{N}} G(u) \, dx.
\end{equation}
Firstly we give a version of Ekeland's principle:
\begin{lem}\label{lem4.3}
Let $n \in \N$ and $\varepsilon>0$. Assume that $\tilde{\gamma}_{n}\in \tilde{\Gamma}_{n}$ satisfies
$$
\max_{\sigma \in D_{n}} \tilde{I}(\tilde{\gamma}(\sigma))\leq \tilde{b}_{n}+\varepsilon.
$$
Then there exists $(\theta, u)\subset \R \times H^{s}_{rad}(\R^{N})$ such that:
\begin{compactenum}[(i)]
\item $\dist_{\R \times H^{s}_{rad}(\R^{N})}((\theta, u), \tilde{\gamma}(D_{n}))\leq 2 \sqrt{\varepsilon}$;
\item $\tilde{I}(\theta, u)\in [b_{n}-\varepsilon, b_{n}+\varepsilon ]$;
\item $\|\nabla \tilde{I}(\theta, u) \|_{\R \times (H^{s}_{rad}(\R^{N}))^{*}}\leq 2 \sqrt{\varepsilon}$.
\end{compactenum}
Here we have used the notations
$$
\dist_{\R \times H^{s}_{rad}(\R^{N})}((\theta, u), A)=\inf_{(t, v)\in A} \sqrt{|\theta-t|^{2}+\|u-v\|^{2}}
$$
for $A\subset \R\times H^{s}_{rad}(\R^{N})$,
and
$$
\nabla \tilde{I}(\theta, u):=\Bigl(\frac{\partial}{\partial \theta}\tilde{I}(\theta, u), \tilde{I}'(\theta, u)  \Bigr).
$$
\end{lem}
\begin{proof}
Since $\tilde{I}(\theta, -u)=\tilde{I}(\theta, u)$ for all $(\theta, u)\in \R\times H^{s}_{rad}(\R^{N})$,
we can deduce that $\tilde{\Gamma}_{n}$ is stable under the pseudo-deformation flow generated by $\tilde{I}(\theta, u)$.
Taking into account $\tilde{b}_{n}=b_{n}>0$, $I(0)=0$ and $\max_{\sigma \in D_{n}}\tilde{I}(0, \gamma_{n}(\sigma))<0$, the proof of Lemma goes as in \cite{J1}.

\end{proof}

\noindent
Thus we can deduce the following result:
\begin{thm}\label{thm4.2}
For any $n\in \N$ there exists a sequence $(\theta_{k}, u_{k})\subset \R \times H^{s}_{rad}(\R^{N})$ such that
\begin{compactenum}[(i)]
\item $\theta_{k} \rightarrow 0$;
\item $\tilde{I}(\theta_{k}, u_{k})\rightarrow b_{n}$;
\item $\tilde{I}'(\theta_{k}, u_{k})\rightarrow 0$ strongly in $(H^{s}_{rad}(\R^{N}))^{*}$;
\item $\frac{\partial}{\partial \theta}\tilde{I}(\theta_{k}, u_{k})\rightarrow 0$.
\end{compactenum}
\end{thm}
\begin{proof}
For any $k\in \N$ there exists $\gamma_{k}\in \Gamma_{n}$ such that
$$
\max_{\sigma \in D_{n}} I(\gamma_{k}(\sigma))\leq b_{n}+\frac{1}{k}.
$$
Since $\tilde{\gamma}_{k}(\sigma)=(0, \gamma_{k}(\sigma))\in \tilde{\Gamma}_{n}$ and $\tilde{b}_{n}=b_{n}$, we have
$$
\max_{\sigma \in D_{n}} \tilde{I}(\tilde{\gamma}_{k}(\sigma))\leq \tilde{b}_{n}+\frac{1}{k}.
$$
By using Lemma \ref{lem4.3}, it follows the existence of $(\theta_{k}, u_{k})\subset \R \times H^{s}_{rad}(\R^{N})$ such that
\begin{align}
&\dist_{\R \times H^{s}_{rad}(\R^{N})}((\theta_{k}, u_{k}), \tilde{\gamma}_{k}(D_{n}))\leq \frac{2}{\sqrt{k}} \label{4.3} \\
&\tilde{I}(\theta_{k}, u_{k})\in \Bigl[b_{n}-\frac{1}{k}, b_{n}+\frac{1}{k} \Bigr] \label{4.4} \\
&\|\nabla \tilde{I}(\theta_{k}, u_{k}) \|_{\R \times (H^{s}_{rad}(\R^{N}))^{*}}\leq \frac{2}{\sqrt{k}}. \label{4.5}
\end{align}
Then, (\ref{4.3}) and $\tilde{\gamma}_{k}(D_{n})\subset \{0\}\times H^{s}_{rad}(\R^{N})$ yields $(i)$. Clearly, $(ii)$ follows by (\ref{4.4}), and  $(iii)$ and $(iv)$ are obtained as consequence of (\ref{4.5}).

\end{proof}

\noindent
Now, we investigate the boundedness and the compactness properties of the sequence $(\theta_{k}, u_{k})\subset \R \times H^{s}_{rad}(\R^{N})$ obtained in Theorem \ref{thm4.2}. More precisely, we are able to prove that
\begin{thm}\label{thmcomp}
Let $(\theta_{k}, u_{k})\subset \R \times H^{s}_{rad}(\R^{N})$ be a sequence satisfying $(i)$-$(iv)$ of Theorem \ref{thm4.2}.\\
Then  we have
\begin{compactenum}[(a)]
\item $(u_{k})$ is bounded in $H^{s}_{rad}(\R^{N})$;
\item $(\theta_{k}, u_{k})$ has a strongly convergent subsequence in $\R \times H^{s}_{rad}(\R^{N})$.
\end{compactenum}
\end{thm}
\begin{proof}
$(a)$ By using $(ii)$ and $(iv)$ of Theorem \ref{thm4.2}, we can see that
$$
\frac{1}{2}e^{(N-2s)\theta_{k}} [u_{k}]_{H^{s}(\R^{N})}^{2}-e^{N \theta_{k}}\int_{\R^{N}} G(u_{k}) dx \rightarrow b_{n}
$$
and
$$
\frac{N-2s}{2}e^{(N-2s)\theta_{k}} [u_{k}]_{H^{s}(\R^{N})}^{2}-Ne^{N \theta_{k}}\int_{\R^{N}} G(u_{k}) dx \rightarrow 0
$$
as $k\rightarrow \infty$.
Therefore we deduce that 
\begin{equation}\label{boundu}
[u_{k}]_{H^{s}(\R^{N})}^{2}\rightarrow \frac{N b_{n}}{s}
\end{equation}
and
\begin{equation*}\label{boundG}
\int_{\R^{N}} G(u_{k}) dx \rightarrow \frac{N-2s}{2s}b_{n}
\end{equation*}
as $k\rightarrow \infty$.\\
By Lemma \ref{lemh}, there exists $C>0$ such that
\begin{equation}\label{ioo}
|h(t)|\leq C|t|^{2^{*}_{s}-1} \mbox{ for all } t\in \R.
\end{equation}
Set 
$$
\varepsilon_{k}=\|\tilde{I}'(\theta_{k}, u_{k})\|_{(H^{s}_{rad}(\R^{N}))^{*}}.
$$
By $(iii)$ of Theorem \ref{thm4.2}, we deduce that $\varepsilon_{k} \rightarrow 0$ as $k \rightarrow \infty$, so we get 
\begin{equation}\label{io}
|\tilde{I}'(\theta_{k}, u_{k})u_{k}|\leq \varepsilon_{k}\|u_{k}\|.
\end{equation}
Taking into account  (\ref{boundu}), (\ref{ioo}), (\ref{io}), $(h3)$ of Lemma \ref{lemh} and by using the Sobolev inequality we have
\begin{align*}
e^{(N-2s)\theta_{k}} [u_{k}]_{H^{s}(\R^{N})}^{2}&+\frac{m}{2}e^{N \theta_{k}}\|u_{k}\|_{L^{2}(\R^{N})}^{2}\\
&\leq e^{N \theta_{k}}\int_{\R^{N}} \frac{m}{2}u^{2}_{k}+g(u_{k})u_{k} \, dx+\varepsilon_{k}\|u_{k}\|\\
&\leq e^{N \theta_{k}}\int_{\R^{N}} h(u_{k})u_{k} \, dx+\varepsilon_{k}||u_{k}||\\
&\leq C e^{N \theta_{k}}\|u_{k}\|_{L^{2^{*}_{s}}(\R^{N})}^{2^{*}_{s}}+\varepsilon_{k}\|u_{k}\|\\
&\leq CC_{*}e^{N \theta_{k}}[u_{k}]_{H^{s}(\R^{N})}^{2^{*}_{s}}+\varepsilon_{k}\|u_{k}\|\\
&\leq CC_{*}C' e^{N \theta_{k}}+\varepsilon_{k}\|u_{k}\|
\end{align*}
from which follows the boundedness of $\|u_{k}\|_{L^{2}(\R^{N})}$.

\noindent
Therefore, recalling \eqref{boundu}, we can deduce that $(u_{k})$ is bounded in $H^{s}_{rad}(\R^{N})$.\\
$(b)$ By $(a)$ we may assume that, up to a subsequence, 
\begin{align}\label{wc}
& u_{k} \rightharpoonup u \mbox{ in } H^{s}_{rad}(\R^{N}), \nonumber \\
& u_{k} \rightarrow u \mbox{ in } L^{q}(\R^{N}) \quad \forall q\in (2, 2^{*}_{s}),\\
& u_{k} \rightarrow u  \mbox{ a.e. } \R^{N}. \nonumber
\end{align}
%By using $(iii)$ of Theorem \ref{thm4.2} and (\ref{wc}) we can see that as $k \rightarrow \infty$
%$$
%e^{(N-2s) \theta_{k}} \iint_{\R^{2N}} \frac{(u_{k}(x)-u_{k}(y))}{|x-y|^{N+2s}} (\varphi(x)-\varphi(y)) \, dx dy-e^{N \theta_{k}}\int_{\R^{N}} g(u_{k}) \varphi \, dx \rightarrow 0
%$$
%for any $\varphi \in C^{\infty}_{0}(\R^{N})$.
%Since $C^{\infty}_{0}(\R^{N})$ is dense in $H^{s}_{rad}(\R^{N})$, it follows that as $k \rightarrow \infty$
By using $(iii)$ of Theorem \ref{thm4.2}, we can see that
\begin{equation}\label{5.8}
e^{(N-2s) \theta_{k}} \iint_{\R^{2N}} \frac{(u_{k}(x)-u_{k}(y))}{|x-y|^{N+2s}} (\varphi(x)-\varphi(y)) \, dx dy-e^{N \theta_{k}}\int_{\R^{N}} g(u_{k}) \varphi \, dx \rightarrow 0
\end{equation}
for any $\varphi \in \mathcal{C}^{\infty}_{0}(\R^{N})$.\\
%for any $\varphi \in H^{s}_{rad}(\R^{N})$.\\
Then $I'(u)\varphi=0$  for any $\varphi \in \mathcal{C}^{\infty}_{0}(\R^{N})$, and by the density of $\mathcal{C}^{\infty}_{0}(\R^{N})$ in $H^{s}_{rad}(\R^{N})$, we get
$I'(u)\varphi=0$  for any $\varphi \in H^{s}_{rad}(\R^{N})$. Moreover, we get
\begin{equation}\label{5.9}
[u]^{2}_{H^{s}(\R^{N})}=\int_{\R^{N}} g(u)u \, dx.
\end{equation}
Now, we observe that (\ref{5.8}) holds for any $\varphi \in H^{s}_{rad}(\R^{N})$ and $(u_{k})$ is bounded in $H^{s}(\R^{N})$. Therefore, taking $\varphi=u_{k}$ in (\ref{5.8}), we have 
\begin{equation*}
e^{(N-2s) \theta_{k}}[u_{k}]^{2}_{H^{s}(\R^{N})}-e^{N \theta_{k}}\int_{\R^{N}} g(u_{k}) u_{k} \, dx=o(1) \mbox{ as } k \rightarrow \infty.
\end{equation*}
Hence we obtain
\begin{align}\label{5.10}
&e^{(N-2s)\theta_{k}} [u_{k}]^{2}_{H^{s}(\R^{N})}+\frac{m}{2}e^{N \theta_{k}} \|u_{k}\|_{L^{2}(\R^{N})}^{2} \nonumber \\
&= e^{N \theta_{k}} \int_{\R^{N}} \frac{m}{2}u_{k}^{2}+g(u_{k})u_{k} \, dx+o(1)\nonumber \\
&=e^{N \theta_{k}} \int_{\R^{N}} h(u_{k})u_{k}dx- e^{N \theta_{k}} \int_{\R^{N}} [h(u_{k})u_{k}- \frac{m}{2} u_{k}^{2}-g(u_{k})u_{k}] \, dx+o(1)\nonumber \\
&=e^{N \theta_{k}} A_{k}-e^{N \theta_{k}}  B_{k}+o(1).
\end{align}
Now, by using $(h6)$ of Lemma \ref{lemh}, we can see that
\begin{equation}\label{5.11}
\lim_{k \rightarrow \infty}A_{k} = \int_{\R^{N}} h(u)u \, dx,
\end{equation}
and by $(h3)$ of Lemma \ref{lemh} and Fatou's Lemma, we get 
\begin{equation}\label{5.12}
\liminf_{k \rightarrow \infty} B_{k}\geq \int_{\R^{N}} \Bigl[h(u)u-\frac{m}{2}u^{2}-g(u)u \Bigr] dx.
\end{equation}
Taking into account \eqref{wc}, (\ref{5.9}), (\ref{5.10}), (\ref{5.11}) and (\ref{5.12}), we can infer that 
\begin{align*}
\|u\|^{2} \leq \limsup_{k \rightarrow \infty} \|u_{k}\|^{2}&=\limsup_{k \rightarrow \infty} \Bigl[e^{(N-2s)\theta_{k}} [u_{k}]^{2}_{H^{s}(\R^{N})}+\frac{m}{2}e^{N\theta_{k}} \|u_{k}\|^{2}_{L^{2}(\R^{N})}\Bigr] \\
&\leq \int_{\R^{N}} \Bigl[\frac{m}{2} u^{2}+g(u)u \Bigr] \, dx\\
&= \|u\|^{2}
\end{align*}
which gives that $u_{k}\rightarrow u$ in $H^{s}_{rad}(\R^{N})$.

\end{proof}

\noindent
Putting together  Theorem \ref{thmcomp}, Lemma \ref{lembn} and (\ref{bn}), we can provide the following multiplicity result.
\begin{thm}\label{multhm}
Under the assumptions of Theorem \ref{thmf}, there exist infinitely many solutions to (\ref{P}).
\end{thm}
\begin{proof}
Fix $n \in \N$, and let $(\theta_{k}, u_{k})\subset \R \times H^{s}_{rad}(\R^{N})$ be a sequence satisfying $(i)$-$(iv)$ of Theorem \ref{thm4.2}.
By using Theorem \ref{thmcomp}, we know that there exists $u_{n}\in H^{s}_{rad}(\R^{N})$ such that $u_{k} \rightarrow u_{n}$ in $H^{s}_{rad}(\R^{N})$ as $k \rightarrow \infty$.
Then $u_{n}$ satisfies
$$
I(u_{n})=\tilde{I}(0, u_{n})=b_{n} \mbox{ and } I'(u_{n})=\tilde{I}'(0, u_{n})=0.
$$
Since $b_{n} \rightarrow \infty$ as $n \rightarrow \infty$, we can conclude that (\ref{P}) admits infinitely many solutions characterized by a mountain-pass argument in $H^{s}_{rad}(\R^{N})$.
\end{proof}

\section{Mountain pass value gives the least energy level}

In this section we prove the existence of a positive solution to (\ref{P}) by using the mountain pass approach developed in the previous section. 
We recall that the existence of a positive ground state solution to (\ref{P}) has been obtained in \cite{CW}.  Here, we give an alternative proof of this fact, and in addition, we show that the least energy solution of (\ref{P}) coincides with the mountain-pass value. This is in clear accordance with the result obtained in the classic framework by Jeanjean and Tanaka in \cite{JT2}.
Moreover, we investigate the symmetry of solutions to (\ref{P}) by using the moving-plane method.\\
The main result of this section can be stated as follows.
\begin{thm}\label{thmone}
Under the assumptions $(g1)$-$(g3)$ there exists a classical positive solution $u$ of (\ref{P}). Moreover $u$ is radially decreasing and $u$ can be characterized by a mountain pass argument in $H^{s}_{rad}(\R^{N})$.
\end{thm}

\noindent 
Let us consider the functionals $I: H^{s}_{rad}(\R^{N})\rightarrow \R$ and $\tilde{I}:\R\times H^{s}_{rad}(\R^{N})\rightarrow \R$ introduced in Section $3$:
$$
I(u)=\frac{1}{2} [u]_{H^{s}(\R^{N})}^{2}-\int_{\R^{N}} G(u) dx 
$$
and
%for any $u\in H^{s}_{rad}(\R^{N})$, and
$$
\tilde{I}(\theta, u)=\frac{1}{2}e^{(N-2s)\theta} [u]_{H^{s}(\R^{N})}^{2}-e^{N \theta}\int_{\R^{N}} G(u) dx. 
$$
%for $(\theta, u)\in \R\times H^{s}_{rad}(\R^{N})$.\\
We define the following minimax values:
\begin{align*}
&b_{mp, r}=\inf_{\gamma\in \Gamma_{r}} \max_{t\in [0, 1]} I(\gamma(t)), \\
&b_{mp}=\inf_{\gamma\in \Gamma} \max_{t\in [0, 1]} I(\gamma(t)), \\
&\tilde{b}_{mp, r}=\inf_{\tilde{\gamma}\in \tilde{\Gamma}_{r}} \max_{t\in [0, 1]} \tilde{I}(\tilde{\gamma}(t)),
\end{align*}
where
\begin{align*}
&\Gamma_{r}=\{\gamma\in \mathcal{C}([0, 1], H^{s}_{rad}(\R^{N})): \gamma(0)=0 \mbox{ and } I(\gamma(1))<0\},\\
&\Gamma=\{\gamma\in \mathcal{C}([0, 1], H^{s}(\R^{N})): \gamma(0)=0 \mbox{ and } I(\gamma(1))<0\}, \\
&\tilde{\Gamma}_{r}=\{\tilde{\gamma}=(\theta, \gamma)\in \mathcal{C}([0, 1], \R\times H^{s}_{rad}(\R^{N})): \gamma\in \Gamma_{r} \mbox{ and } \theta(0)=0=\theta(1)\}.
\end{align*}

\noindent
Then, we can see that the following result holds:
\begin{lem}\label{Byeonlem}
$$
b_{mp}=b_{mp, r}=b_{1}=\tilde{b}_{mp, r}.
$$
\end{lem}
\begin{proof}
Arguing as in the proof of Lemma \ref{lembn}, we have $b_{mp, r}=\tilde{b}_{mp, r}$. 
Furthermore, it follows by the above definitions that $b_{mp}\leq b_{mp, r}\leq b_{1}$. In order to conclude the proof of Lemma, it is enough to show that $b_{1}\leq b_{mp}$.\\
Firstly, we show that 
\begin{equation}\label{pepsi}
b_{mp}=\bar{b},
\end{equation}
where
\begin{align*}
\bar{b}=\inf_{\gamma\in \bar{\Gamma}} \max_{t\in [0, 1]} I(\gamma(t)) 
\end{align*}
and
\begin{align*}
\bar{\Gamma}=\{\gamma\in \mathcal{C}([0, 1], H^{s}(\R^{N})): \gamma(0)=0 \mbox{ and } \gamma(1)=\gamma_{1}(1)\}.
\end{align*}
Here $\gamma_{1}$ is the path appearing in Lemma \ref{lem4.1} (with $n=1$). As observed in \cite{BL1, BL2}, we may assume that $\gamma_{1}$ satisfies the following properties: $\gamma(1)(x)\geq 0$ for any $x\in \R^{N}$, $\gamma_{1}(|x|)=\gamma(1)(x)$ and $r\mapsto \gamma_{1}(1)(r)$ is piecewise linear and non increasing.\\
Since $I(\gamma_{1}(1))<0$, it is clear that $\bar{\Gamma}\subset \Gamma$ which gives $b_{mp}\leq \bar{b}$. Now, we show that $b_{mp}\geq \bar{b}$. \\
For this purpose, it is suffices to prove that $B=\{u\in H^{s}(\R^{N}): I(u)<0\}$ is path connected. We proceed as in \cite{Byeon}.   
Take $u_{1}, u_{2}\in B$ and assume that $\supp(u_{1})\cup \supp(u_{2})\subset B_{R}$ for some $R>0$.
In fact, if $\phi\in \mathcal{C}^{\infty}_{0}(\R^{N})$ is such that $\phi(x)=0$ for $|x|\geq 2$ and $\phi(x)=1$ for $|x|\leq 1$, denoted by $\phi_{t}(x)=\phi(t x)$ for $t>0$, we can see that there exists $t_{0}>0$ such that $I(t u_{i}+(1-t)\phi_{t_{0}}u_{i})<0$ for any $t\in [0, 1]$ and $i=1, 2$.\\
Now, let us denote by $u_{i}^{t}(x)=u_{i}(\frac{x}{t})$ for $t>0$ and $i=1, 2$.\\
Then 
\begin{align*}\label{I(t)}
I(u_{i}^{t})=\frac{t^{N-2s}}{2}[u_{i}]_{H^{s}(\R^{N})}^{2}-t^{N} \int_{\R^{N}} G(u_{i}) \,dx
\end{align*}
so we get 
\begin{align*}
\frac{d}{dt} I(u_{i}^{t})= N t^{N-1} \Bigl(t^{-2s} \frac{1}{2} [u_{i}]_{H^{s}(\R^{N})}^{2}- \int_{\R^{N}} G(u_{i}) \,dx\Bigr)-st^{N-2s-1} [u_{i}]_{H^{s}(\R^{N})}^{2}. 
\end{align*}
In particular, for $t\geq 1$ we have
$$
\frac{d}{dt} I(u_{i}^{t})\leq N t^{N-1} I(u_{i})-s t^{N-2s-1} [u_{i}]_{H^{s}(\R^{N})}^{2}<0.
$$
Since $u_{i}\in B$, we deduce that $\int_{\R^{N}} G(u_{i}) dx>0$, 
so we can see that $I(u_{i}^{t})\rightarrow -\infty \mbox{ as } t \rightarrow \infty$.\\
Hence, we can choose $t_{1}>1$ and $t_{2}>1$ such that 
\begin{equation}\label{ut1}
I(u_{1}^{t_{1}})\leq -1-\max_{t\in [0, 1]} |I(t u_{2})|
\end{equation}
and
\begin{equation}\label{ut2}
I(u_{2}^{t_{2}})\leq -1-\max_{t\in [0, 1]} |I(t u_{1}^{t_{1}})|.
\end{equation}
Now, we define $u_{3}^{\theta}(x)=u_{1}^{t_{1}}(x+2\theta R(t_{1}+t_{2})e_{1})$ for $\theta\in [0, 1]$, where $e_{1}=(1, 0, \dots, 0)$.
It is clear that for all $\theta\in [0, 1]$
$$
I(u_{3}^{\theta})=I(u_{3}^{0})=I(u_{1}^{t_{1}})<0.
$$
Let us observe that $\supp(u^{1}_{3})\cap \supp(u_{2}^{t})=\emptyset$ for all $t\in [1, t_{2}]$, so by using (\ref{ut1}), we deduce that
$$
I(u_{3}^{1}+\theta u_{2})=I(u_{3}^{1})+I(\theta u_{2})<0
$$
for all $\theta\in [0, 1]$.
We also note that $I(u_{3}^{1}+u_{2}^{t})=I(u_{3}^{1})+I(u_{2}^{t})$ for $t\in [1, t_{2}]$, $I(u_{2}^{t})<0$ for all $t\in [1, t_{2}]$, and $I(\theta u_{3}^{1}+u_{2}^{t_{2}})=I(\theta u_{1}^{t_{1}})+I(u_{2}^{t_{2}})<0$ for $\theta\in [0, 1]$ in view of (\ref{ut2}).
Then, considering the following paths 
\begin{align*}
&\gamma_{1}(\theta)=u_{1}^{\theta} \mbox{ for } \theta\in [1, t_{1}], \\
&\gamma_{2}(\theta)=u_{3}^{\theta} \mbox{ for } \theta\in [0, 1],\\
&\gamma_{3}(\theta)=u_{3}^{1}+\theta u_{2} \mbox{ for } \theta\in [0, 1],\\
&\gamma_{4}(\theta)=u_{3}^{1}+u_{2}^{\theta} \mbox{ for } \theta\in [1, t_{2}],\\
&\gamma_{5}(\theta)=\theta u_{3}^{1}+u_{2}^{t_{2}} \mbox{ for } \theta\in [0, 1],\\
&\gamma_{6}(\theta)=u_{2}^{\theta} \mbox{ for } \theta\in [1, t_{2}]
\end{align*}
we can find a path connecting $u_{1}$ and $u_{2}$ on which $I$ is negative. Therefore, $B$ is path connected and $b_{mp}\geq \bar{b}$. This implies that \eqref{pepsi} holds. 

Now, observing that $I(u)=I(-u)$, $\gamma_{1}(1)\in H^{s}_{rad}(\R^{N})$ and $\gamma(-t)=-\gamma(t)$ for any $\gamma\in \Gamma_{1}$, we aim to show that 
\begin{align}\label{pepsi1}
\bar{b}=\inf_{\gamma\in \bar{\Gamma}_{r}} \max_{t\in [0, 1]} I(\gamma(t))(=b_{1})
\end{align}
where 
\begin{align*}
\bar{\Gamma}_{r}=\{\gamma\in \mathcal{C}([0, 1], H^{s}_{rad}(\R^{N})): \gamma(0)=0 \mbox{ and } \gamma(1)=\gamma_{1}(1)\}.
\end{align*}
It is clear that by definition $\bar{b}\leq b_{1}$. In what follows, we show that $\bar{b}\geq b_{1}$.\\
Take $\eta\in \bar{\Gamma}$ and we set $\gamma(t):=|\eta(t)|$. Clearly, $\gamma\in \mathcal{C}([0, 1], H^{s}(\R^{N}))$. Moreover, recalling that $G$ is even and by using the fact that $[|u|]_{H^{s}(\R^{N})}\leq [u]_{H^{s}(\R^{N})}$ for any $u\in H^{s}(\R^{N})$ (this inequality follows by $||x|-|y||\leq |x-y|$ for any $x, y\in \R$), we can see that  for any  $t\in [0,1]$ it holds
\begin{align}\begin{split}\label{pepsi100}
I(\gamma(t))&=\frac{1}{2}[|\eta(t)|]_{H^{s}(\R^{N})}^{2}-\int_{\R^{N}} G(|\eta(t)|)\, dx \\
&=\frac{1}{2}[|\eta(t)|]_{H^{s}(\R^{N})}^{2}-\int_{\R^{N}} G(\eta(t))\, dx \\
&\leq \frac{1}{2}[\eta(t)]_{H^{s}(\R^{N})}^{2}-\int_{\R^{N}} G(\eta(t))\, dx=I(\eta(t)).
\end{split}\end{align}
Moreover, by using $\gamma_{1}(1)\geq 0$, we can see that $\gamma(1)=|\eta(1)|=|\gamma_{1}(1)|=\gamma_{1}(1)$, so that $\gamma\in \bar{\Gamma}$.
Now, we denote by $\gamma^{*}(t)$ the Schwartz symmetrization $\gamma(t)$. Then, $\gamma^{*}(t)\in H^{s}_{rad}(\R^{N})$, and by using the continuity of $G$ and the fractional Polya-Szeg\"o inequality $[u^{*}]_{H^{s}(\R^{N})}\leq [u]_{H^{s}(\R^{N})}$ (see \cite{AL}), we can deduce that $I(\gamma^{*}(t))\leq I(\gamma(t))$ for any $t\in [0, 1]$. Since the rearrangement map is continuous everywhere on $H^{s}(\R^{N})$ (see \cite{AL}), we have $\gamma^{*}\in \mathcal{C}([0, 1], H^{s}_{rad}(\R^{N}))$, and by using the fact that $(\gamma_{1}(1))^{*}=\gamma_{1}(1)$, we can deduce that $\gamma^{*}\in \bar{\Gamma}_{r}$.\\
Therefore, in view of \eqref{pepsi100}, we have
\begin{align*}
b_{1}\leq \max_{t\in [0, 1]} I(\gamma^{*}(t))\leq \max_{t\in [0, 1]} I(\gamma(t)) \leq \max_{t\in [0, 1]} I(\eta(t)),
\end{align*}
which implies that $b_{1}\leq \bar{b}$. This shows that \eqref{pepsi1} is satisfied.
Putting together \eqref{pepsi} and \eqref{pepsi1}, we can infer that $b_{mp}=\bar{b}=b_{1}$.
This ends the proof of lemma.

\end{proof}

\noindent
Then, we are able to prove the following result.
\begin{thm}\label{exthm}
There exists a positive solution to (\ref{P}) such that $I(u)=b_{mp,r}$. Moreover, 
for any non-trivial solution $v$ to (\ref{P}), we have $b_{mp,r}\leq I(v)$. This means that $u$ is the least energy solution to (\ref{P}) and that $b_{mp,r}$ is the least energy level.
\end{thm}
\begin{proof}
We argue as in Section $3$. Let $(\gamma_{k})\subset \Gamma_{r}$ be a sequence such that
\begin{equation}\label{6.55}
\max_{t\in [0, 1]} I(\gamma_{k}(t))\leq b_{mp, r}+\frac{1}{k}.
\end{equation}
Since $||x|-|y||\leq |x-y|$ for any $x, y\in \R$ and $G$ is even, we can see that $I(|u|)\leq I(u)$. Then, 
%$I$ and $\tilde{I}$ are even in $u$ and $\tilde{I}(\theta, |u|)\leq I(\theta, u)$, so
we may assume that $\gamma_{k}\in \Gamma_{r}$ in (\ref{6.55}) satisfies 
$$
\gamma_{k}(t)(x)\geq 0 \mbox{ for all } t\in \R, x\in \R^{N}.
$$
In view of Lemma \ref{Byeonlem}, we can deduce that
\begin{equation}\label{6.5}
\max_{t\in [0, 1]} \tilde{I}(0, \gamma_{k}(t))\leq \tilde{b}_{mp, r}+\frac{1}{k}=b_{1}+\frac{1}{k}.
\end{equation}
Therefore, by applying Lemma \ref{lem4.3} to $\tilde{I}$ and $(0, \gamma_{k})$, we can find $(\theta_{k}, u_{k})\subset \R\times H^{s}_{rad}(\R^{N})$ such that as $k\rightarrow \infty$
\begin{compactenum}[$(i)$]
\item $\dist_{\R\times H^{s}_{rad}(\R^{N})}((\theta_{k}, u_{k}), \{0\}\times \gamma_{k}([0, 1]))\rightarrow 0$;
\item $\tilde{I}(\theta_{k}, u_{k})\rightarrow b_{1}$;
\item $\|\nabla \tilde{I}(\theta_{k}, u_{k}) \|_{\R \times (H^{s}_{rad}(\R^{N}))^{*}}\rightarrow 0$.
\end{compactenum}
Taking into account $(ii)$, we can see that $\theta_{k}\rightarrow 0$ as $k \rightarrow \infty$, and
$$
\|u^{-}_{k}\|_{H^{s}(\R^{N})}\leq \dist_{\R\times H^{s}_{rad}(\R^{N})}((\theta_{k}, u_{k}), \{0\}\times \gamma_{k}([0, 1]))\rightarrow 0.
$$
Proceeding as in the proof of Theorem \ref{thmcomp}, we can prove that $u_{k}\rightarrow u$ in $H^{s}(\R^{N})$, $I'(u)=0$ and $I(u)=b_{1}$, for some $u\in H^{s}_{rad}(\R^{N})$ such that $u\geq 0$, $u\not\equiv 0$. 
Since $u\in \mathcal{C}^{0, \beta}(\R^{N})\cap L^{\infty}(\R^{N})$ (see Lemma \ref{regularitylem} before), we can apply the Harnack inequality \cite{CS1, FallFelli} to obtain $u>0$ in $\R^{N}$. This ends the proof of the first statement of theorem.

Now, let $v$ is a nontrivial solution to (\ref{P}) and we define 
\begin{equation*}
\gamma(t)(x)= 
\left\{
\begin{array}{ll}
v(\frac{x}{t}) &\mbox{ for } t>0\\
0 &\mbox{ for } t=0.   
\end{array}
\right.
\end{equation*}
Then we can see  
\begin{align*}
&\|\gamma(t)\|^{2}_{H^{s}(\R^{N})}=t^{N-2s}[v]^{2}_{H^{s}(\R^{N})}+t^{N}\|v\|^{2}_{2} \nonumber \\
&I(\gamma(t))= \frac{t^{N-2s}}{2} [v]^{2}_{H^{s}(\R^{N})}- t^{N} \int_{\R^{N}} G(v) \,dx. 
\end{align*}
It is clear that $\gamma\in C([0, \infty), H^{s}_{rad}(\R^{N}))$. \\
Since every weak solution of \eqref{P} satisfies the Pohozaev Identity \cite{CW}, we have
\begin{equation*}
\int_{\R^{N}}G(v) dx=\frac{N-2s}{2N} [v]^{2}_{H^{s}(\R^{N})}>0, 
\end{equation*}
so we can infer that
$$
\frac{d}{dt} I(\gamma(t))>0 \mbox{ for } t\in (0, 1) \,
\mbox{ and } \, \frac{d}{dt} I(\gamma(t))<0 \mbox{ for } t>1.
$$
Thus, after a suitable scale change in $t$, there exists a path $\gamma(t): [0, 1] \rightarrow H^{s}_{rad}(\R^{N})$ such that $\gamma(0)=0$, $I(\gamma(1))<0$, $v\in \gamma([0,1])$ and
$$
\max_{t\in [0,1]} I(\gamma(t)) =I(v). 
$$
Therefore, $b_{mp,r}=b_{1}$ is corresponding to a positive least energy solution to \eqref{P}.

\end{proof}

\noindent
Now, we study regularity and symmetry of solutions of (\ref{P}) in the case of ``positive mass".
\begin{lem}\label{regularitylem}
Let $u$ be a positive solution to $(-\Delta)^{s} u=g(u)$ in $\R^{N}$, where $g\in \mathcal{C}^{1, \alpha}(\R)$ is such that $g'(0)<0$ and $\lim_{|t|\rightarrow \infty} \frac{g(t)}{|t|^{2^{*}_{s}-1}}=0$.
Then $u\in \mathcal{C}^{2, \beta}(\R^{N})\cap L^{\infty}(\R^{N})$ is radially symmetric and strictly decreasing about some point in $\R^{N}$.
\end{lem}
\begin{proof}
The proof of regularity of $u$ is given in \cite{CW}. However, for reader's convenience, we give a proof of it here. Since $|g(t)|\leq C(1+|t|^{2^{*}_{s}-1})$ for any $t\in \R$ and $u\geq 0$, we can apply Proposition $5.1.1$ in \cite{DMV}, to deduce that $u\in L^{\infty}(\R^{N})$. Taking into account $g\in \mathcal{C}^{1, \alpha}(\R)$, we can use Lemma $4.4$ in \cite{CS1} to infer that $u\in \mathcal{C}^{2, \beta}(\R^{N})$ for some $\beta\in (0,1)$ depending only on $s$ and $\alpha$. Observing that $u\in \mathcal{C}^{0, \beta}(\R^{N})\cap L^{2^{*}_{s}}(\R^{N})$, we have $u(x)\rightarrow 0$ as $|x|\rightarrow \infty$.\\
%By using a Brezis-Kato argument, we know that if $u$ is a weak solution to $(-\Delta)^{s} u=a(x) u$ in $\R^{N}$, with $a\in L^{\frac{N}{2s}}(\R^{N})$, then $u\in L^{q}(\R^{N})$ for any $q\in [2, \infty)$. Since $|g(t)|\leq C(|t|+|t|^{2^{*}_{s}-1})$, we deduce that $g(u(x))\in L^{q}(\R^{N})$ for any $q\in [2, \infty)$, so we get $u\in W^{s, q}(\R^{N})$ for any $q<\infty$. By using Sobolev embedding, we can infer that $u\in \mathcal{C}^{0, \gamma}(\R^{N})\cap L^{\infty}(\R^{N})$ for some $\gamma\in (0,1)$. In particular $u$ decays to zero at infinity. 
%By using Lemma $4.4$ in \cite{CS1}, we deduce that $u\in \mathcal{C}^{2, \beta}(\R^{N})$ for some $\beta\in (0,1)$.\\
Now, we show that $u$ is radially symmetric and strictly decreasing about some point in $\R^{N}$.
In order to achieve our aim, we follow some ideas developed in \cite{DMPS, FW}.

For $\lambda\in \R$, we define the following sets
$$
\Sigma_{\lambda}=\{(x_{1}, x')\in \R^{N}: x_{1}>\lambda\}
$$
and
$$
T_{\lambda}=\{(x_{1}, x')\in \R^{N}: x_{1}=\lambda\}
$$
and we denote by $u_{\lambda}(x)=u(x^{\lambda})$, where $x^{\lambda}=(2\lambda-x_{1}, x')$.\\
We divide the proof into three steps.\\
Step $1$: We show that $\lambda_{0}=\sup\{\lambda: u_{\lambda}\leq u \mbox{ in } \Sigma_{\lambda}\}$ is finite.\\
Let us define
\begin{equation*}
w(x):= 
\left\{
\begin{array}{ll}
(u_{\lambda}-u)^{+}(x) &\mbox{ if } x\in \Sigma_{\lambda}\\
(u_{\lambda}-u)^{-}(x) &\mbox{ if } x\in \Sigma^{c}_{\lambda}. 
\end{array}
\right. 
\end{equation*}
Since $w(x)=-w(x_{\lambda})$ for any $x\in \R^{N}$, we get $\int_{\R^{N}} |w|^{2}\, dx=2\int_{\Sigma_{\lambda}} |w|^{2}\, dx$. By using the fact that if $\varphi\in \mathcal{D}^{s,2}(\R^{N})$ then $\varphi_{\lambda}\in \mathcal{D}^{s, 2}(\R^{N})$, we can deduce that $w\in H^{s}(\R^{N})$.

Hence, we can use $w$ as test function in the weak formulations of \eqref{P} and $(-\Delta)^{s}u_{\lambda}=g(u_{\lambda})$ in $\R^{N}$, and subtracting them, we obtain 
\begin{align*}
\iint_{\R^{2N}} \frac{[(u_{\lambda}(x)-u(x))-(u_{\lambda}(y)-u(y))]}{|x-y|^{N+2s}}(w(x)-w(y)) \,dx dy=\int_{\R^{N}} (g(u_{\lambda})-g(u))w \,dx.
\end{align*}
Recalling that $x^{-}=-(-x)^{+}$ for any $x\in \R$, we can note that
\begin{align*}
&\int_{\R^{N}} (g(u_{\lambda})-g(u))w \,dx=\int_{\Sigma_{\lambda}} (g(u_{\lambda})-g(u))(u_{\lambda}-u)^{+} \,dx+\int_{\Sigma^{c}_{\lambda}} (g(u_{\lambda})-g(u))(u_{\lambda}-u)^{-} \,dx \\
&=2 \int_{\Sigma_{\lambda}} (g(u_{\lambda})-g(u))(u_{\lambda}-u)^{+} \,dx.
\end{align*}
Therefore, we have
\begin{align}\label{FW3}
\iint_{\R^{2N}} \frac{[(u_{\lambda}(x)-u(x))-(u_{\lambda}(y)-u(y))]}{|x-y|^{N+2s}}(w(x)-w(y)) \,dx dy=2\int_{\Sigma_{\lambda}} (g(u_{\lambda})-g(u))(u_{\lambda}-u)^{+} \,dx.
\end{align}
%\begin{align}\label{FW3}
%\int_{\Sigma_{\lambda}} [(-\Delta)^{s}u_{\lambda}-(-\Delta)^{s}u] (u_{\lambda}-u)^{+} dx=\int_{\Sigma_{\lambda}} (g(u_{\lambda})-g(u))(u_{\lambda}-u)^{+} dx.
%\end{align}
Let us observe that
\begin{align}\label{FW4}
&\iint_{\R^{2N}} \frac{[(u_{\lambda}(x)-u(x))-(u_{\lambda}(y)-u(y))]}{|x-y|^{N+2s}}(w(x)-w(y)) \,dx dy \nonumber\\
&=\iint_{\R^{2N}} \frac{|w(x)-w(y)|^{2}}{|x-y|^{N+2s}} \,dx dy+\iint_{\R^{2N}} \frac{\mathcal{G}(x, y)}{|x-y|^{N+2s}}\, dx dy
\end{align}
where
$$
\mathcal{G}(x, y):=\left(((u_{\lambda}(x)-u(x))-(u_{\lambda}(y)-u(y))-(w(x)-w(y))\right) (w(x)-w(y)).
$$
Arguing as in the proof of Theorem $1.6$ in \cite{DMPS} (see formula $(3.29)$ there), we can show that 
$$
\iint_{\R^{2N}} \frac{\mathcal{G}(x, y)}{|x-y|^{N+2s}}\, dx dy\geq 0,
$$
which together with \eqref{FW3} and \eqref{FW4} yields
\begin{align}\label{FW5}
\iint_{\R^{2N}} \frac{|w(x)-w(y)|^{2}}{|x-y|^{N+2s}} \,dx dy\leq 2\int_{\Sigma_{\lambda}} (g(u_{\lambda})-g(u))(u_{\lambda}-u)^{+} \,dx.
\end{align}
Now, we know that $g'(0)<0$, so we can find $\delta>0$ such that $g'(t)<0$ for all $|t|<\delta$.
Since $u(x)\rightarrow 0$ as $|x|\rightarrow \infty$,  there exists $R>0$ such that $0<u(x)<\delta$ for all $|x|>R$.\\ 
Then, if $x\in \Sigma_{\lambda}$ and $\lambda<-R$, we deduce that
$u_{\lambda}(x)<\delta$, so we get
\begin{align}\label{FW6}
\int_{\Sigma_{\lambda}} (g(u_{\lambda})-g(u))(u_{\lambda}-u)^{+} dx\leq 0
\end{align}
for all $\lambda<-R$.\\
%On the other hand, by using Theorem \ref{Sembedding}, we find a constant $C>0$ such that
%\begin{align}\label{FW7}
%&\iint_{\R^{2N}} \frac{|w(x)-w(y)|^{2}}{|x-y|^{N+2s}} \,dx dy \geq C \left(\int_{\R^{N}} |w|^{\frac{2N}{N-2s}}\, dx\right)^{\frac{N-2s}{N}}. 
%\end{align}
%On the other hand, recalling that $(x-y)(x^{+}-y^{+})\geq |x^{+}-y^{+}|^{2}$ for any $x, y\in \R$, we have 
%\begin{align}\label{FW2}
%&\int_{\Sigma_{\lambda}} [(-\Delta)^{s}u_{\lambda}-(-\Delta)^{s}u] (u_{\lambda}-u)^{+} dx \nonumber\\
%&=\int_{\Sigma_{\lambda}} \, dx\int_{\R^{N}} \frac{[(u_{\lambda}-u)(x)-(u_{\lambda}-u)(y)]}{|x-y|^{N+2s}}[(u_{\lambda}-u)^{+}(x)-(u_{\lambda}-u)^{+}(y)] \, dy  \nonumber\\ 
%&\geq \int_{\Sigma_{\lambda}} \, dx\int_{\R^{N}} \frac{|(u_{\lambda}-u)^{+}(x)-(u_{\lambda}-u)^{+}(y)|^{2}}{|x-y|^{N+2s}} \, dy  \nonumber\\
%&=\int_{\Sigma_{\lambda}} [(-\Delta)^{\frac{s}{2}} (u_{\lambda}-u)^{+}|^{2} dx=\int_{\Sigma_{\lambda}} [(-\Delta)^{\frac{s}{2}} w|^{2} dx.
%\end{align}
Putting together \eqref{FW5} and \eqref{FW6} we deduce that $w$ is constant on $\R^{N}$. Since $w=0$ on $T_{\lambda}$, we can infer that $u_{\lambda}\leq u$ in $\Sigma_{\lambda}$  for all  $\lambda<-R$.
This implies that $\lambda_{0}\geq -R$.
Since $u$ decays at infinity, there exists $\lambda_{1}$ such that $u(x)<u_{\lambda_{1}}(x)$ for some $x\in \Sigma_{\lambda_{1}}$. Hence $\lambda_{0}$ is finite.\\
Step $2$: We prove that $u\equiv u_{\lambda_{0}}$ in $\Sigma_{\lambda_{0}}$. We assume by contradiction that $u\neq u_{\lambda_{0}}$ and $u\geq u_{\lambda_{0}}$ in $\Sigma_{\lambda_{0}}$. Suppose that there exists $x_{0}\in \Sigma_{\lambda_{0}}$ such that $u_{\lambda_{0}}(x_{0})=u(x_{0})$.\\
Then we can see that
\begin{equation}\label{4.11}
(-\Delta)^{s} u_{\lambda_{0}}(x_{0}) -(-\Delta)^{s}u(x_{0}) = g(u_{\lambda_{0}}(x_{0}))-g(u(x_{0}))=0. 
\end{equation}
On the other hand, observing that $|x_{0}-y|<|x_{0}-y_{\lambda_{0}}|$ for all $y\in \Sigma_{\lambda_{0}}$ and recalling that $u\neq u_{\lambda_{0}}$, we get
\begin{align*}
(-\Delta)^{s} &u_{\lambda_{0}}(x_{0}) - (-\Delta^{s})u(x_{0}) =-\int_{\R^{N}} \frac{u_{\lambda_{0}}(y) -u(y)}{|x_{0}-y|^{N+2s}}\, dy \\
&=-\int_{\Sigma_{\lambda_{0}}} (u_{\lambda_{0}}(y) - u(y)) \left( \frac{1}{|x_{0}-y|^{N+2s}} -\frac{1}{|x_{0} -y_{\lambda_{0}}|^{N+2s}}\right) \, dy>0,
\end{align*}
which contradicts (\ref{4.11}). Therefore $u>u_{\lambda_{0}}$ in $\Sigma_{\lambda_{0}}$.\\
In order to complete Step $2$, we only need to prove that $u\geq u_{\lambda}$ in $\Sigma_{\lambda}$ continues to hold when $\lambda>\lambda_{0}$ is close to $\lambda_{0}$.\\
%$\lambda_{0}<\lambda<\lambda_{0} +\varepsilon$, where $\varepsilon>0$ is small. \\
Fix $\varepsilon>0$ whose value will be chosen later, and take $\lambda \in (\lambda_{0}, \lambda_{0}+\varepsilon)$. Let $P=(\lambda, 0)$ and $B(P,R)$ be the ball centered at $P$ and with radius $R>1$ to be chosen later.\\
%such that $u(x)<\delta$ for $x\in B^{c}(P, R)$.
Set $\tilde{B}= \Sigma_{\lambda}\cap B(P,R)$. By using $w$ as test function in the equations for $u$ and $u_{\lambda}$, we have
\begin{align}\label{4.12}
&\iint_{\R^{2N}} \frac{[(u_{\lambda}(x)-u(x))-(u_{\lambda}(y)-u(y))]}{|x-y|^{N+2s}}(w(x)-w(y)) \,dx dy \nonumber\\
&=2 \int_{\Sigma_{\lambda}} (g(u_{\lambda}) -g(u))(u_{\lambda}-u)^{+} \, dx \nonumber \\
&=2\int_{\tilde{B}} (g(u_{\lambda})-g(u))(u_{\lambda}-u)^{+} \, dx+2\int_{\Sigma_{\lambda} \setminus\tilde{B}} (g(u_{\lambda}) -g(u))(u_{\lambda}-u)^{+} \, dx.
%\left( \int_{\Sigma_{\lambda}} |w|^{\frac{2N}{N-2s}} dx \right)^{\frac{N-2s}{N}}\leq C \int_{\Sigma_{\lambda}} (f(u_{\lambda}) -f(u))(u_{\lambda}-u)^{+} \, dx. 
\end{align}
Since $u$ goes to zero at infinity, there exists $R_{0}>0$ such that $u(x)<\delta$ for all $x\in B^{c}_{R_{0}}(0)$.
Choose $R>0$ such that $\Sigma_{\lambda} \setminus\tilde{B}\subset B^{c}(P, R)\subset B^{c}_{R_{0}}(0)$.
Therefore, $u_{\lambda}(x)<\delta$ for all $x\in \Sigma_{\lambda} \setminus\tilde{B}$, and we can see that 
$$
\int_{\Sigma_{\lambda} \setminus\tilde{B}} (g(u_{\lambda}) -g(u))(u_{\lambda}-u)^{+} \, dx\leq 0.
$$
This and (\ref{4.12}) yield
\begin{align*}
%\label{biofilm}
\iint_{\R^{2N}} \frac{[(u_{\lambda}(x)-u(x))-(u_{\lambda}(y)-u(y))]}{|x-y|^{N+2s}}(w(x)-w(y)) \,dx dy\leq 2\int_{\tilde{B}} (g(u_{\lambda}) -g(u))(u_{\lambda}-u)^{+} \, dx.
\end{align*}
Now, arguing as in Step $1$, and by using Theorem \ref{Sembedding}, we can find a positive constant $C_{0}>0$ such that
\begin{align*}
%\label{mel}
&\iint_{\R^{2N}} \frac{[(u_{\lambda}(x)-u(x))-(u_{\lambda}(y)-u(y))]}{|x-y|^{N+2s}}(w(x)-w(y)) \,dx dy \nonumber \\
&\geq \iint_{\R^{2N}} \frac{|w(x)-w(y)|^{2}}{|x-y|^{N+2s}} \,dx dy \nonumber \\
&\geq C_{0} \|w\|_{L^{2^{*}_{s}}(\Sigma_{\lambda})}^{2}.
\end{align*}
Then, we deduce that
\begin{align}\label{biofilm}
 \|w\|_{L^{2^{*}_{s}}(\Sigma_{\lambda})}^{2}\leq 2C_{0}^{-1}\int_{\tilde{B}} (g(u_{\lambda}) -g(u))(u_{\lambda}-u)^{+} \, dx.
\end{align}
By using H\"older inequality, we can see that
\begin{align}\label{4.13}
&2C_{0}^{-1}\int_{\tilde{B}} (g(u_{\lambda}) -g(u)) (u_{\lambda}-u)^{+} dx \leq C \int_{\tilde{B}} |w|^{2} \chi_{\supp(u_{\lambda} -u)^{+}} dx \nonumber \\
&=C|\tilde{B} \cap \supp(u_{\lambda}-u)^{+}|^{\frac{2s}{N}} \left(\int_{\tilde{B}} |w|^{\frac{2N}{N-2s}} dx\right)^{\frac{N-2s}{N}}. 
\end{align}
Taking into account $u>u_{\lambda_{0}}$ in $\Sigma_{\lambda_{0}}$ and the continuity of $u$, we can see that $u>u_{\lambda}$ on any compact $K\subset \Sigma_{\lambda}$, for $\lambda$ close to $\lambda_{0}$. This means that  $\supp(u_{\lambda}-u)^{+}$ is small in $\tilde{B}$ for $\lambda$ close to $\lambda_{0}$, so we can choose $\varepsilon>0$ such that 
\begin{equation}\label{ceps}
C|\tilde{B}\cap \supp(u_{\lambda}-u)^{+}|^{\frac{2s}{N}}<\frac{1}{2}.
\end{equation}
Putting together (\ref{biofilm}), (\ref{4.13}) and (\ref{ceps}) we deduce that $w=0$ in $\Sigma_{\lambda}$, which gives a contradiction. This concludes the proof of Step $2$. \\
Step $3$: By translation, we may say that $\lambda_{0}=0$. 
Thus, $u(-x_{1}, x')\geq u(x_{1}, x')$ for $x_{1}\geq 0$. A similar argument in $x_{1}<0$ shows that 
$u(-x_{1}, x')\leq u(x_{1}, x')$ for $x_{1}\geq 0$, so we have $u(-x_{1}, x')= u(x_{1}, x')$ for $x_{1}\geq 0$. Using the same procedure in any arbitrary direction, we deduce that $u$ is radially symmetric. \\
%Thus, we have that $u$ is symmetric  about $x_{1}$-axis, that is $u(-x_{1}, x')=u(x_{1}, x')$. Using the same argument in any arbitrary direction, we conclude that $u$ is radially symmetric. \\
Now, we prove that $u(r)$ is strictly decreasing in $r>0$.
Let us consider $0<x_{1}<\tilde{x}_{1}$ and let $\lambda=\frac{x_{1}+\tilde{x}_{1}}{2}$.
Then, as proved above, we have $w_{\lambda}(x)>0$ for $x\in \Sigma_{\lambda}$. 
Hence
$$
0<w_{\lambda}(\tilde{x}_{1}, x')=u_{\lambda}(\tilde{x}_{1}, x')-u(\tilde{x}_{1}, x')=u(x_{1}, x')-u(\tilde{x}_{1}, x')
$$  
that is $u(x_{1}, x')>u(\tilde{x}_{1}, x')$. This shows the monotonicity of $u$ in $x_{1}$ for $x_{1}>0$.
%Using the radial symmetry of $u$, we deduce the monotonicity of $u$.

\end{proof}

\noindent
We conclude this Section giving the proofs of Theorem \ref{thmone} and Theorem \ref{thmf}:

\begin{proof}[Proof of Theorem \ref{thmone}]
It follows by Theorem \ref{exthm} and Lemma \ref{regularitylem}. 

\end{proof} 

\begin{proof}[Proof of Theorem \ref{thmf}]
Putting together Theorem \ref{multhm}, Theorem \ref{thmone} and Lemma \ref{regularitylem}, we obtain the desired result.

\end{proof}

\section{The zero mass case: Proof of Theorem \ref{thmf2}}

This last section is devoted to the proof of the existence of a nontrivial solution of (\ref{P}) in the zero mass case. 
Let $\e_{0}=\frac{G(\xi_{0})}{\xi_{0}^{2}}>0$ and let us define $g_{\e}(t)=g(t)-\e t$ with $\e \in (0, \e_{0}]$.
Then $g_{\e}$ satisfies the assumption of Theorem \ref{thmone}, so we know that for any $\e \in (0, \e_{0}]$ there exists $u_{\e}\in H^{s}_{rad}(\R^{N})$ positive and radially decreasing  in $r=|x|$, such that $I_{\e}(u_{\e})=b_{\e}$ and $I_{\e}'(u_{\e})=0$, where the mountain pass minimax value $b_{\e}$ is defined as $b_{\e}=\inf_{\gamma\in \Gamma_{\e}} \max_{t\in [0, 1]}I_{\e}(\gamma(t))$ with $\Gamma_{\e}=\{\gamma\in \mathcal{C}([0, 1], H^{s}_{rad}(\R^{N})): \gamma(0)=0, I_{\e}(\gamma(1))<0\}$. 
Since $u_{\e}$ satisfies the Pohozaev identity \eqref{pepsiboom}, we can see that $\frac{s}{N}[u_{\e}]_{H^{s}(\R^{N})}^{2}=b_{\e}$. 
Now we prove that it is possible to estimate $b_{\e}$ from above independently of $\e$:
\begin{lem}\label{thmindependent}
There exists $b_{0}>0$ such that $0<b_{\e}\leq b_{0}$ for all $\e \in (0, \e_{0}]$.
\end{lem}
\begin{proof}
For $R>1$, define 
\begin{equation*}
w_{R}(x)= 
\left\{
\begin{array}{ll}
\xi_{0} &\mbox{ for } |x|\leq R\\
\xi_{0}(R+1-r) &\mbox{ for } r=|x|\in [R, R+1]\\
0 &\mbox{ for } |x|\geq R+1. 
\end{array}
\right. 
\end{equation*}
It is clear that $w_{R}\in H^{s}(\R^{N})$. By the definition of $w_{R}$ and  $g_{\e}(t)=g(t)-\e t$, we deduce that for all $\e \in (0, \e_{0}]$
\begin{align*}
\int_{\R^{N}}G_{\e}(w_{R}) dx & \geq \int_{\R^{N}}G_{\e_{0}}(w_{R}) \,dx \\
&= G_{\e_{0}}(\xi_{0})|B_{R}|+\int_{\{r\leq |x|\leq r+1\}} G_{\e_{0}}(\xi_{0}(R+1-|x|)) \,dx \\
&\geq G_{\e_{0}}(\xi_{0})|B_{R}| - |B_{R+1}-B_{R} | \max_{t\in [0, \xi_{0}]} |G_{\epsilon_{0}}(t)| \\
&=\frac{\pi^{\frac{N}{2}}}{\Gamma(\frac{N}{2}+1)}[G_{\e_{0}}(\xi_{0})R^{N}-\max_{t\in [0, \xi_{0}]} |G_{\e_{0}}(t)| ((R+1)^{N}-R^{N})] \\
&\geq \frac{\pi^{\frac{N}{2}}}{\Gamma(\frac{N}{2}+1)} [G_{\e_{0}}(\xi_{0})-\max_{t\in [0, \xi_{0}]} |G_{\e_{0}}(t)| ((1+\frac{1}{R})^{N}-1)]R^{N} 
\end{align*}
so there exists $\bar{R}>0$ (independent of $\e$) such that $\int_{\R^{N}}G_{\e}(w_{R}) \,dx>0$ for all $R\geq \bar{R}$ and $\e \in (0, \e_{0}]$.
Then, by setting $w_{t}(x)=w_{\bar{R}}(\frac{x}{t})$ we have for any $\e \in (0, \e_{0}]$
\begin{equation}
I_{\e}(w_{t})=\frac{t^{N-2s}}{2}[w_{\bar{R}}]_{H^{s}(\R^{N})}^{2}-t^{N} \int_{\R^{N}} G_{\e}(w_{\bar{R}}) \,dx \rightarrow -\infty
\end{equation}
as $t \rightarrow \infty$.
Hence there exists $\tau>0$ such that $I_{\e}(w_{\tau})<0$ for any $\e \in (0, \e_{0}]$. We put $\bar{e}(x)=w_{\tau}(x)$.
%Then we define 
%$$
%\Gamma_{\epsilon}=\{\gamma\in C([0, 1], H^{s}_{r}(\R^{N})): \gamma(0)=0 \mbox{ and } \gamma(1)=\bar{e}\}.
%$$
Therefore, 
$$
I_{\e}(u_{\e})=b_{\e}\leq \sup_{\e \in (0, \e_{0}]}\max_{t\in [0, 1]} I_{\e}(t \bar{e})=:b_{0}
$$
for any $\e \in (0, \e_{0}]$.
\end{proof}

\noindent
In view of Lemma \ref{thmindependent}, we can infer that 
\begin{equation}\label{bound}
[u_{\e}]_{H^{s}(\R^{N})}^{2}= \frac{N}{s} b_{\e}\leq \frac{N}{s} b_{0} \mbox{ for all } \e \in (0, \e_{0}],
\end{equation}
and as a consequence, we can assume that as $\e \rightarrow 0$
\begin{align}\begin{split}\label{convergenze}
&u_{\e} \rightharpoonup u \mbox{ in } \mathcal{D}^{s, 2}_{rad}(\R^{N}) \\
&u_{\e} \rightarrow u \mbox{ in } L^{q}_{loc}(\R^{N}), \mbox{ for any } q\in [2, 2^{*}_{s})\\
&u_{\e} \rightarrow u \mbox{ a.e. in } \R^{N}. 
\end{split}\end{align} 
Since $u_{\e}$ is a weak solution to (\ref{P}), we know that 
\begin{equation}\label{wfeps}
\langle u_{\e}, \varphi \rangle_{\mathcal{D}^{s, 2}(\R^{N})} = \int_{\R^{N}} [g(u_{\e})- \e u_{\e}] \varphi \, dx,
\end{equation}
for any $\varphi \in \mathcal{C}^{\infty}_{0}(\R^{N})$. 
Taking into account  (\ref{bound}), (\ref{convergenze}), $(h1)$-$(h2)$, $\mathcal{D}^{s, 2}(\R^{N}) \subset L^{2^{*}_{s}}(\R^{N})$ we can apply the first part of Lemma \ref{strauss} with $P(t)=g(t)$, $Q(t)=|t|^{2^{*}_{s}-1}$, $v_{\e}=u_{\e}$, $v=u$ and $w=\varphi$ to pass to the limit in (\ref{wfeps}) as $\e \rightarrow 0$. Then we obtain
\begin{equation*}
\langle u, \varphi \rangle_{\mathcal{D}^{s, 2}(\R^{N})} = \int_{\R^{N}} g(u) \varphi  \, dx,
\end{equation*}
for any $\varphi \in \mathcal{C}^{\infty}_{0}(\R^{N})$. 
This means that $u$ is a weak solution to $(-\Delta)^{s}u=g(u)$ in $\R^{N}$. Now, we only need to prove that $u\not\equiv 0$.\\
We recall that $u_{\e}> 0$ and satisfies the Pohozaev identity \eqref{pepsiboom} with $G$ replaced by $G_{\e}$, that is
\begin{equation}\label{pepsiboom1}
\frac{N-2s}{2}[u_{\e}]_{H^{s}(\R^{N})}^{2}= N \int_{\R^{N}} \left(G(u_{\e}) - \frac{\e}{2} u_{\e}^{2}\right) \, dx. 
\end{equation}
Now, by using $(h_{1})$-$(h_{2})$, there exists $c>0$ such that
\begin{equation}\label{catarella}
G(t)\leq c|t|^{2^{*}_{s}} \mbox{ for all } t\in \R.
\end{equation}
Then, in view of \eqref{pepsiboom1} and \eqref{catarella}, we get 
\begin{align*}
[u_{\e}]_{H^{s}(\R^{N})}^{2}&\leq \frac{2N}{N-2s}\frac{\e}{2}\|u_{\e}\|^{2}_{L^{2}(\R^{N})}+[u_{\e}]_{H^{s}(\R^{N})}^{2}\\
&=\frac{2N}{N-2s}\int_{\R^{N}} G(u_{\e}) \,dx\\
&\leq \frac{2N}{N-2s} c\int_{\R^{N}} |u_{\e}|^{2^{*}_{s}} \,dx \\
&\leq C[u_{\e}]_{H^{s}(\R^{N})}^{2^{*}_{s}}
\end{align*}
so we can see that 
\begin{equation}\label{absurd}
[u_{\e}]_{H^{s}(\R^{N})}\geq c>0,
\end{equation}
for some $c$ independent of $\e$.\\
Next, we argue by contradiction, and we assume that $u=0$.
%Our aim is to show that by using the facts $u_{\epsilon}$ converges weakly to zero in ${D}_{r}^{s, 2}(\R^{N})$, $u_{\epsilon}$ is radial decreasing  and $[u_{\epsilon}]^{2}\leq \frac{N}{s}b_{0}$, we can deduce that 
We begin proving that
\begin{equation}\label{gepsilonzero}
\int_{\R^{N}} G^{+}(u_{\e}) \,dx \rightarrow 0 \mbox{ as } \e \rightarrow 0.
\end{equation}
%where $G_{+}=\max\{G, 0\}$ and $G_{-}=\max\{G, 0\}$.
%which gives a contradiction in view of (\ref{bnl}) and $\limsup_{\epsilon \rightarrow 0} [u_{\epsilon}]^{2}\geq c>0$.\\
By using Lemma \ref{radlem} with $t=2^{*}_{s}$, Theorem \ref{Sembedding} and \eqref{bound},  
%$\|u_{\e}\|_{L^{2^{*}_{s}}(\R^{N})}\leq S_{*} [u_{\e}]_{H^{s}(\R^{N})}\leq C$ 
we can see that, for any $x\in \R^{N}\setminus \{0\}$ and $\e \in (0, \e_{0}]$
\begin{align}\label{11.6}
|u_{\e}(x)|&\leq \left(\frac{N}{\omega_{N-1}}\right)^{\frac{1}{2^{*}_{s}}} |x|^{-\frac{N}{2^{*}_{s}}} \|u_{\e}	\|_{L^{2^{*}_{s}}(\R^{N})}\nonumber\\
&\leq \left(\frac{N}{\omega_{N-1}}\right)^{\frac{1}{2^{*}_{s}}} |x|^{-\frac{N}{2^{*}_{s}}} S_{*} [u_{\e}]_{H^{s}(\R^{N})}\nonumber \\
&\leq \left(\frac{N}{\omega_{N-1}}\right)^{\frac{1}{2^{*}_{s}}} |x|^{-\frac{N}{2^{*}_{s}}} S_{*} \sqrt{\frac{N}{s} b_{0}}= C  |x|^{-\frac{N}{2^{*}_{s}}}, 
\end{align}
where $C$ is independent of $\e$. 
By (\ref{11.6}), for any $\delta>0$ there exists $R>0$ such that $|u_{\e}(x)|\leq \delta$ for $|x|\geq R$. Moreover by assumption $(h_1)$, for any $\eta>0$ there exists $\delta>0$ such that 
\begin{equation*}
G^{+}(t)\leq \eta |t|^{2^{*}_{s}} \mbox{ for } |t|\leq \delta. 
\end{equation*} 
Hence for large $R$ we obtain the following estimate 
\begin{align*}
\int_{\R^{N}\setminus B_{R}} G^{+}(u_{\e})\, dx &\leq \eta \int_{\R^{N}\setminus B_{R}} |u_{\e}|^{2^{*}_{s}} dx\\
&\leq C \eta [u_{\e}]_{H^{s}(\R^{N})}^{2^{*}_{s}} \leq C \eta,
\end{align*}
uniformly in $\e\in (0, \e_{0}]$. On the other hand, by assumption $(h_2)$, for any $\eta>0$ there exists a constant $C_{\eta}$ such that 
\begin{equation*}
G^{+}(t) \leq \eta |t|^{2^{*}_{s}}+ C_{\eta} \mbox{ for all } t\in \R. 
\end{equation*} 
Now, we fix $\Omega \subset \R^{N}$ with sufficiently small measure $|\Omega|<\frac{\eta}{C_{\eta}}$. Then, for any $\e \in (0, \e_{0}]$
\begin{equation*}
\int_{\Omega} G^{+}(u_{\e})\, dx \leq \eta \int_{\Omega} |u_{\e}|^{2^{*}_{s}} \,dx + C_{\eta} |\Omega| \leq C\eta. 
\end{equation*}
By applying Vitali's convergence Theorem, we can deduce that (\ref{gepsilonzero}) is satisfied. \\
Therefore, by using the Pohozaev Identity \eqref{pepsiboom1} and the facts $G=G^{+}+G^{-}$ and $G^{-}\leq 0$, we have
\begin{align}\label{bnl}
[u_{\e}]_{H^{s}(\R^{N})}^{2}&\leq [u_{\e}]_{H^{s}(\R^{N})}^{2} -\frac{2N}{N-2s}\int_{\R^{N}} G^{-}(u_{\e}) \,dx+\frac{2N}{N-2s}\int_{\R^{N}} \frac{\e}{2}|u_{\e}|^{2} \,dx \nonumber \\
&=\frac{2N}{N-2s}\int_{\R^{N}} G^{+}(u_{\e}) \,dx 
\end{align}
which together with (\ref{gepsilonzero}) yields $0<c\leq \limsup_{\e \rightarrow 0} [u_{\e}]_{H^{s}(\R^{N})}^{2}=0$. This gives a contradiction in view of (\ref{absurd}).
Then $u\in \mathcal{D}^{s, 2}_{rad}(\R^{N})$ is a nontrivial weak solution to (\ref{P}). Arguing as in the proof of Lemma \ref{regularitylem}, we can see that $u\in \mathcal{C}^{0, \beta}(\R^{N})\cap L^{\infty}(\R^{N})$. From the Harnack inequality \cite{CS1, FallFelli} we conclude that $u>0$ in $\R^{N}$.


\begin{thebibliography}{777}

\bibitem{AL}
F.J. Almgren and E. Lieb,
{\it Symmetric decreasing rearrangement is sometimes continuous},
J. Am. Math. Soc. 2, (1989), 683--773.

\bibitem{ASM1}
C. O. Alves, M. A.S. Souto and M. Montenegro,
{\it Existence of a ground state solution for a nonlinear scalar field equation with critical growth},
Calc. Var. {\bf 43} (2012), 537--554. 

\bibitem{ASM2}
C. Alves, A. S. Souto, and M. Montenegro,
{\it Existence of solution for two classes of elliptic problems in $\mathbb{R}^{N}$ with zero mass}, 
J. Differential Equations {\bf 252} (2012), no. 10, 5735--5750.

\bibitem{AFM}
A. Ambrosetti, V. Felli, and A. Malchiodi,
{\it Ground states of nonlinear Schr\"odinger equations with potentials vanishing at infinity}
J. Eur. Math. Soc. (JEMS) {\bf 7} (2005), no. 1, 117--144.

\bibitem{A}
V. Ambrosio,
{\it Periodic solutions for a pseudo-relativistic Schr\"odinger equation},
Nonlinear Anal. {\bf 120} (2015), 262--284.

\bibitem{A1}
V. Ambrosio, 
{\it Ground states solutions for a non-linear equation involving a pseudo-relativistic Schr\"odinger operator}, 
J. Math. Phys. 57 (2016), no. 5, 051502, 18 pp.

%\bibitem{A1}
%V. Ambrosio,
%{\it Ground states for superlinear fractional Schro\"odinger equations in $\R^{N}$},
%(2016) to appear on  Ann. Acad. Sci. Fenn. Math.; doi : 10.5186/aasfm.2016.4147.

\bibitem{A2}
V. Ambrosio,
{\it Zero mass case for a fractional Berestycki-Lions type problem}, 
Adv. Nonlinear Anal. doi: 10.1515/anona-2016-0153.
%Preprint. arXiv:1602.05726 .

\bibitem{AP1}
A. Azzollini and A. Pomponio,
{\it On a ``zero mass'' nonlinear Schr\"odinger equation},
Adv. Nonlinear Stud. {\bf 7} (2007), no. 4, 599--627.

\bibitem{ADP}
A. Azzollini, P. d'Avena and A. Pomponio,
{\it Multiple critical points for a class of nonlinear functionals},
Ann. Math. Pura Appl. {\bf 190} (2011), 507--523.  

%\bibitem{BGM1}
%V. Benci, C. R. Grisanti and A.M. Micheletti
%{\it Existence and non existence of the ground state solution for the nonlinear Schr\"odinger equations with $V(\infty)=0$},
%Topol. Methods Nonlinear Anal. {\bf 26} (2005), no. 2, 203--219. 

\bibitem{BGM2}
V. Benci, C. R. Grisanti and A.M. Micheletti,
{\it Existence of solutions for the nonlinear Schr\"odinger equations with $V (\infty) = 0$}, 
Contributions to non-linear analysis, 53--65, Progr. Nonlinear Differential Equations Appl., 66, Birkh\"auser, Basel, 2006.

\bibitem{BM}
V. Benci and A. M. Micheletti
{\it Solutions in exterior domains of null mass nonlinear field equations},
 Adv. Nonlinear Stud. {\bf 6} (2006), no. 2, 171--198.

\bibitem{BL1}
H. Berestycki and P.L. Lions, 
{\it Nonlinear scalar field equations. I. Existence of a ground state}, 
Arch. Rational Mech. Anal. {\bf 82} (1983), no. 4, 313--345.

\bibitem{BL2}
H. Berestycki and P.L. Lions, 
{\it Nonlinear scalar field equations. 2. Existence of infinitely many solutions}, 
Arch. Rational Mech. Anal. {\bf 82} (1983), no. 4, 347--375.

\bibitem{BGK}
H. Berestycki, T. Gallou\"et and O. Kavian,
{\it Equations de Champs scalaires euclidiens non lin{{\'e}}aires dans le plan},
C. R. Acad. Sci. Paris S{{\'e}}r. I Math. {\bf 297} (1983) 307--310.

%\bibitem{BKW}
%{\sc P. Biler, G. Karch and W.A.Woyczynski},
%{\em Critical nonlinearity exponent and self-similar asymptotics for Levy conservation laws}, 
%Ann. Inst. H. Poincar\'{e} Anal. Non Lin\'{e}aire {\bf 18} (2001), 613--637.

\bibitem{Byeon}
J. Byeon, 
{\it Singularly perturbed nonlinear Dirichlet problems with a general nonlinearity},
Trans. Amer. Math. Soc. {\bf 362}, no. 4, (2010), 1981--2001.

%\bibitem{Cabsolmor}
%{\sc X. Cabr{{\'e}} and J. Sol{{\`a}}-Morales},
%{\em Layer solutions in a half-space for boundary reactions},
%Comm. Pure Appl. Math. {\bf 58} (2005),1678--1732.

\bibitem{CS1}
X. Cabr{{\'e}} and Y.Sire,
{\it Nonlinear equations for fractional Laplacians I: regularity, maximum principles, and Hamiltonian estimates},
Ann. Inst. H. Poincar\'{e} Anal. Non Lin\'{e}aire {\bf 31} (2014), 23--53.

%\bibitem{CV}
%{\sc L. Caffarelli and A. Vasseur},
%{\em Drift diffusion equations with fractional
%diffusion and the quasi-geostrophic equation},
%Ann. of Math. {\bf 2} 171 (2010), no.3, 1903--1930.

\bibitem{CS}
L.A. Caffarelli and L.Silvestre,
{\it An extension problem related to the fractional Laplacian},
Comm. Partial Differential Equations {\bf 32} (2007),1245--1260.

%\bibitem{CMS}
%{\sc R. Carmona, W. C. Masters, B. Simon},
%{\em Relativistic Schrodinger operators; Asymptotic behaviour of the eigenfunctions},
%J. Funct. Anal. {\bf 91} (1990), 117--142. 

\bibitem{CW}
X. J. Chang and Z.Q. Wang, 
{\it Ground state of scalar field equations involving fractional Laplacian with general nonlinearity}, 
Nonlinearity {\bf 26}, 479--494 (2013).

%\bibitem{Cheng}
%M. Cheng, 
%{\it Bound state for the fractional Schr\"odinger equation with unbounded potential},
%J. Math. Phys. {\bf 53}, 043507 (2012).

%\bibitem{CT}
%{\sc R. Cont and P. Tankov}, 
%{\em Financial Modelling with Jump Processes}, 
%Chapman and Hall/CRC Financ. Math. Ser., Chapman
%and Hall/CRC, Boca Raton, FL, (2004).

\bibitem{CTN1}
V. Coti Zelati and M. Nolasco, 
{\it Existence of ground states for nonlinear, pseudo-relativistic Schr\"odinger equations},
Atti Accad. Naz. Lincei Cl. Sci. Fis. Mat. Natur. Rend. Lincei (9) Mat. Appl 22 (1), (2011) 51--72.

\bibitem{CTN2}
V. Coti Zelati and M. Nolasco, 
{\it Ground states for pseudo-relativistic Hartree equations of critical type},
Rev. Mat. Iberoam. 29(4), (2013) 1421--1436.

\bibitem{DMV}
S. Dipierro, M. Medina, and E. Valdinoci, 
{\it Fractional elliptic problems with critical growth in the whole of $\R^{N}$}, 
(2015). arXiv:1506.01748vl.

\bibitem{DMPS}
S. Dipierro, L. Montoro, I. Peral, and B. Sciunzi, 
{\it Qualitative properties of positive solutions to nonlocal critical problems involving the Hardy-Leray potential},
Calc. Var. Partial Differential Equations {\bf 55} (2016), no. 4, Paper No. 99, 29 pp.

\bibitem{DPV}
E. Di Nezza, G. Palatucci, E. Valdinoci, 
{\it Hitchhiker's guide to the fractional Sobolev spaces}, 
Bull. Sci. math. {\bf 136} (2012), 521--573.

\bibitem{DPPV}
S. Dipierro, G. Palatucci, and E. Valdinoci, 
{\it Existence and symmetry results for a Schr\"odinger type problem involving the fractional Laplacian}, 
Le Matematiche (Catania) {\bf 68}, 201--216 (2013).


\bibitem{FallFelli}
M. Fall and V. Felli, 
{\it Unique continuation properties for relativistic Schr\"odinger operators with a singular potential}, Discrete Contin. Dyn. Syst. {\bf 35} (2015), no. 12, 5827--5867.

\bibitem{FQT}
P. Felmer, A. Quaas and J.Tan,
{\it Positive solutions of the nonlinear {S}chr{\"o}dinger equation with the fractional {L}aplacian},
Proc. Roy. Soc. Edinburgh Sect. A {\bf 142} ( 2012 ), 1237--1262.

\bibitem{FW}
P. Felmer and Y. Wang
{\it Radial symmetry of positive solutions to equations involving the fractional Laplacian},
Commun. Contemp. Math. {\bf 16} (2014), no. 1, 1350023, 24 pp.


\bibitem{FS}
G. M. Figueiredo and G. Siciliano,
{\it A multiplicity result via Ljusternick-Schnirelmann category and Morse theory for a fractional Schr\"odinger equation in $\R^{N}$},
NoDEA Nonlinear Differential Equations Appl. {\bf 23} (2016), no. 2, Art. 12, 22 pp.


\bibitem{FLS}
R. L. Frank, E. Lenzmann and L. Silvestre, 
{\it Uniqueness of radial solutions for the fractional Laplacian},
Comm. Pure Appl. Math. {\bf 69} (2016), no. 9, 1671--1726.
 
%\bibitem{FL}
%R. Frank and E. Lenzmann, 
%{\it Uniqueness and nondegeneracy of ground states for $(-\Delta)^{s} Q + Q - Q^{\alpha +1}=0$ in $\R$},
%e-print arXiv:1009.4042.

%\bibitem{FJLL}
%J.Fr{{\"o}}hlich, B.Jonsson, G.Lars and E.Lenzmann,
%{\it Boson stars as solitary waves},
%Comm. Math. Phys. {\bf 274} (2007), 1--30.

\bibitem{G}
B. Gidas, 
{\it Bifurcation Phenomena in Mathematical Physics and Related Topics},
 Proceedings of the NATO Advanced Study Institute held at Carg{{\`e}}se, June 24-July 7, 1979 (C. Bardos and D. Bessis eds.), NATO Advanced Study Institute Series. Ser. C, Mathematical and Physical Sciences 54, 1980.
 
 \bibitem{GNN} 
B. Gidas, W.M. Ni and L. Nirenberg, 
{\it Symmetry and related properties via the maximum principle},
Comm. Math. Phys. {\bf 68}, n. 3 (1979), 209--243.

\bibitem{HIT}
J. Hirata, N. Ikoma and K. Tanaka,
{\it Nonlinear scalar field equations in $\R^{N}$: mountain pass and symmetric mountain pass approaches},
Topol. Methods Nonlinear Anal. {\bf 35} (2010), 253--276.


\bibitem{N.Ikoma}
N. Ikoma,
{\it Existence of solutions of scalar field equations with fractional operator}, 
Preprint arXiv:1603.04006

\bibitem{J1}
L. Jeanjean, 
{\it Existence of solutions with prescribed norm for semilinear elliptic equations},
Nonlinear Anal. {\bf 28} (1997), no. 10, 1633--1659.

\bibitem{J}
L. Jeanjean, 
{\it On the existence of bounded Palais-Smale sequences and application to a Landesman-Lazer type problem set on $\R^{N}$},
Proc. R. Soc. Edinb., Sect. A, Math., {\bf 129} (1999), 787--809.

\bibitem{JT1}
L. Jeanjean  and K. Tanaka,
{\it A positive solution for an asymptotically linear elliptic problem on $\R^{N}$ autonomous at infinity},
ESAIM Control Optim. Calc. Var. {\bf 7} (2002), 597--614. 

\bibitem{JT2}
L. Jeanjean  and K. Tanaka ,
{\it A remark on least energy solutions in $\R^{N}$},
 Proc. Am. Math. Soc. {\bf 131} (2003), 2399--2408.

\bibitem{Laskin1}
N. Laskin,
{\it Fractional quantum mechanics and L\`evy path integrals}, 
Phys. Lett. A {\bf 268} (2000), 298--305.

\bibitem{Laskin2}
N. Laskin, 
{\it Fractional Schr\"odinger equation}, 
Phys. Rev. E {\bf 66} (2002), 056108.

\bibitem{LY1}
E.H.Lieb and H.T. Yau,
{\it The {C}handrasekhar theory of stellar collapse as the limit of quantum mechanics},
Comm. Math. Phys., {\bf 112} (1987),  147--174.

\bibitem{Lions}
P. L. Lions,
{\it Symetri\'e et compacit\'e dans les espaces de Sobolev},
J.Funct.Anal. {\bf 49}, (1982), 315-334.

%\bibitem{PSV}
%G. Palatucci, O. Savin, and E. Valdinoci,
%{\it Local and global minimizers for a variational energy involving a fractional norm},
%Ann. Mat. Pura Appl. (4) {\bf 192} (2013), no. 4, 673--718.

\bibitem{MBR}
G. Molica Bisci and V. Radulescu,
{\it Ground state solutions of scalar field fractional Schr\"odinger equations},
Calc. Var. Partial Differential Equations {\bf 54} (2015), n. 3, 2985--3008.

\bibitem{MRS}
G. Molica Bisci, V. R\u{a}dulescu, and R. Servadei,
{\it Variational Methods for Nonlocal Fractional Problems},
with a Foreword by Jean Mawhin, {\em Encyclopedia of Mathematics and its Applications}, {\em Cambridge University Press}, \textbf{162} Cambridge, 2016.

%\bibitem{PSY}
%K. Perera, M. Squassina and Y. Yang,
%{\it Critical fractional p-Laplacian problems with possibly vanishing potentials},
%J. Math. Anal. Appl. {\bf 433} (2016), no. 2, 818--831.


\bibitem{P}
P. Pucci, M. Xiang and B. Zhang, 
{\it Multiple solutions for nonhomogeneous Schr\"odinger-Kirchhoff type equations involving the fractional $p$-Laplacian in $\R^{N}$}, 
Calc. Var. Partial Differential Equations {\bf 54} (2015), 2785--2806.

\bibitem{Rab}
P.H. Rabinowitz,
{\it Minimax methods in critical point theory with applications to differential equations}
CBMS Regional Conference Series in Mathematics {\bf 65}, (1986). 

\bibitem{Secchi1}
S. Secchi, 
{\it Ground state solutions for nonlinear fractional Schr\"odinger equations in $\R^{N}$}, 
J. Math. Phys. {\bf 54} (2013), 031501.

\bibitem{Secchi2}
S. Secchi, 
{\it Perturbation results for some nonlinear equations involving fractional operators},
Differ. Equ. Appl. {\bf 5} (2013), no. 2, 221--236.

%\bibitem{Silvestre}
%L. Silvestre,
%{\it Regularity of the obstacle problem for a fractional power of the Laplace operator},
%Comm. Pure Appl. Math., {\bf 60} (2007), no. 1, 67--112.




%\bibitem{sv1}
%R. Servadei and E. Valdinoci, 
%{\it Mountain Pass solutions for non--local elliptic operators}, 
%J. Math. Anal. Appl. {\bf 389} (2012), 887--898.

%\bibitem{sv2} 
%R. Servadei and E. Valdinoci, 
%{\it Variational methods for non--local operators of elliptic type}, 
%Discrete Contin. Dyn. Syst. {\bf33} (2013), 2105--2137.

\bibitem{Struwe}
M. Struwe, 
{\it Variational methods. Applications to nonlinear partial differential equations and Hamiltonian systems}, Fourth edition. 34. Springer-Verlag, Berlin, 2008. xx+302 pp.










\end{thebibliography}
\end{document}